\documentclass[11pt]{amsart}
\usepackage{hyperref}
\usepackage{amsfonts,amssymb,amsmath,amsthm,amscd,euscript,array,mathrsfs}
\usepackage[all]{xy}
\usepackage{bbm}
\usepackage{booktabs}
\usepackage{color}

\usepackage[foot]{amsaddr}
\input{mymacros-3.sty} 

\renewcommand{\arraystretch}{1.2}

\title{The Capelli eigenvalue problem for  Lie superalgebras}

\author{Siddhartha Sahi $^{\mathrm{\lowercase{a}}}$}
\address{ $^{\mathrm{\lowercase{a}}}$Department of Mathematics,    
    Rutgers University,
    110 Frelinghuysen Rd, 
    Piscataway, NJ 08854-8019,
    USA
  }
  \email{sahi@math.rutgers.edu}

  \author{Hadi Salmasian$^{\mathrm{\lowercase{b}}}$}

  \address{
  $^{\mathrm{\lowercase{b}}}$ Department of Mathematics and Statistics,
    University of Ottawa,
    585 King Edward Ave,
    Ottawa, Ontario,
    Canada K1N 6N5%
  }

  \email{hadi.salmasian@uottawa.ca}

\author{Vera Serganova$^{\mathrm{\lowercase{c}}}$}
\address{
$^{\mathrm{\lowercase{c}}}$ Department of Mathematics,
    University of California at Berkeley,
    969 Evans Hall,
    Berkeley, CA 94720,
    USA
    } 
\email{serganov@math.berkeley.edu}

\thanks{\textsc{Acknowledgements.} The research of Siddhartha Sahi was partially supported by a Simons Foundation grant (509766), of Hadi Salmasian by NSERC Discovery Grants (RGPIN-2013-355464 and RGPIN-2018-04044), and of Vera Serganova by an NSF Grant  (1701532). This work was initiated during the Workshop on Hecke Algebras and Lie Theory, which was held at the University of Ottawa. The first and the second named authors thank the National Science Foundation (DMS-162350), the Fields Institute, and the University of Ottawa for funding this workshop.\\[-3mm]}

\begin{document}

\begin{abstract}
For a  finite dimensional unital complex simple Jordan superalgebra $J$, the Tits-Kantor-Koecher construction
yields a 3-graded Lie superalgebra $\gflat\cong \gflat(-1)\oplus\gflat(0)\oplus\gflat(1)$, such that $\gflat(-1)\cong J$. Set $V:=\gflat(-1)^*$ and $\g g:=\gflat(0)$. 
   In most cases, the space  $\sP(V)$ of superpolynomials on $V$ is a completely reducible and multiplicity-free representation of $\g g$, and  there exists a direct sum decomposition $\sP(V):=\bigoplus_{\lambda\in\Omega}V_\lambda$, where $\left(V_\lambda\right)_{\lambda\in\Omega}$ is a family of irreducible $\g g$-modules parametrized by a set of partitions $\Omega$. In these cases, 
one can define  a natural basis $\left(D_\lambda\right)_{\lambda\in\Omega}$ 
of ``Capelli operators'' for the algebra $\sPD(V)^{\g g}$ of $\g g$-invariant superpolynomial differential operators on $V$.
In this paper we complete the solution to  the Capelli eigenvalue problem, which  asks for the determination of the scalar 
$c_\mu(\lambda)$ by which $D_\mu$  acts on $V_\lambda$. 

We  associate a restricted root system $\mathit{\Sigma}$ to the symmetric pair $(\g g,\g k)$ that corresponds to $J$, which  is either a deformed root system of type $\mathsf{A}(m,n)$ or a  root system of type $\mathsf{Q}(n)$.  
We prove  a necessary and sufficient condition on the structure of $\mathit{\Sigma}$ for  $\sP(V)$ to be completely reducible and multiplicity-free.
When $\mathit{\Sigma}$ satisfies the latter condition  we obtain an explicit formula for the eigenvalue 
$c_\mu(\lambda)$, in terms of Sergeev-Veselov's shifted super Jack polynomials when $\mathit{\Sigma}$ is of type $\mathsf{A}(m,n)$, and Okounkov-Ivanov's factorial Schur $Q$-polynomials when $\mathit{\Sigma}$ is of type $\mathsf{Q}(n)$. 
Along the way, we prove 
that the  natural map from the centre of the enveloping algebra of $\g g$ into $\sPD(V)^{\g g}$ is surjective in all cases except when $J\cong\mathit{F}$, where $\mathit{F}$ is the
10-dimensional exceptional Jordan superalgebra. 

\end{abstract}

\maketitle

\section{Introduction and main results}
\label{sec:introduction}
Let $J$ be a finite dimensional unital complex simple Jordan superalgebra (for the classification of these Jordan superalgebras see \cite{KacJ} and \cite{CanKac}). 
The Tits-Kantor-Koecher construction  (see
Appendix \ref{Sec-TKKconst}) associates
\footnote{We remark that $\gflat$ is a slight modification of the 
simple Lie superalgebra 
 that is constructed from $J$ by the
Kantor functor (see Remark \ref{rmk-reasongflat}).}
 to  $J$ a Lie
superalgebra $\gflat$ together with an imbedded $\g{sl}_2$-triple 
$\g s:=\spn_\C\{h,e,f\}$ where 
\begin{equation}
\label{Eqhefrel}
[h,e]=2e,\ [e,f]=h,\text{ and }[h,f]=-2f.
\end{equation} 
Following Kac (see \cite{KacJ} and 
\cite{CanKac}), we consider the grading of $\gflat$ by the eigenspaces of $\ad(-\frac{1}{2}h)$. Then we obtain a  \textquotedblleft short grading\textquotedblright
\begin{equation*}
\gflat\cong \gflat(-1)\oplus \gflat(0)\oplus \gflat(1),
\end{equation*}%
where $\gflat(-1)\cong J$ and $e$ is the identity element of $J$. Set $%
\mathfrak{g}:=\gflat(0)$ and $\mathfrak{k}:=\mathrm{stab}_{\mathfrak{g}}\left(
e\right)$. Then $\left( \mathfrak{g,k}\right) $ is a symmetric pair, and in
fact $\mathfrak{k=g}^{\Theta }$ where $\Theta:=\Ad_w$, for 
$w\in \mathrm{PSL}_2(\C)$ representing the nontrivial element of the Weyl group,
defined as in 
\eqref{involw}. 

Set $V:=\gflat(-1)^{\ast }$, 
where $\gflat(-1)^{\ast }$ denotes the dual of the $\g g$-module $\gflat%
(-1)$. 
 Let $\sP(V)$ denote the superalgebra of superpolynomials on $V$. Note that there is a canonical $\g g$-module isomorphism $\sP(V)\cong \sS(J)$, where $\sS(J)$ denotes the symmetric algebra of the $\Z/2$-graded vector space $J$. 
In most cases
(see Theorem \ref{thm:compl-red}),
the $\g g$-module $\sP(V)$ is  
completely reducible  
and multiplicity-free, and  
the irreducible summands of $\sP(V)$ are parametrized by a set of partitions $\Omega$, i.e., \[
\sP(V)\cong\bigoplus_{\lambda\in\Omega}
V_\lambda,
\]
where the $V_\lambda$ are mutually non-isomorphic irreducible finite dimensional $\g g$-modules.  In these cases, 
to each $\lambda$ one can associate 
a
\emph{Capelli operator}
\footnote{The classical Capelli operator appears as a special case of the operators $D_\lambda$. For this reason, we call the $D_\lambda$ the \emph{Capelli operators}.}
  $D_\lambda\in \sPD(V)^\g g$, where 
$\sPD(V)^\g g$ denotes the algebra of $\g g$-invariant superpolynomial differential operators on $V$.
Indeed the family 
$\left(D_\lambda\right)_{\lambda\in\Omega}$ forms a basis of $\sPD(V)^\g g$ (see Remark \ref{rem:CapelliEven}). 

From Schur's Lemma it follows that each operator $D_\mu$ acts on $V_\lambda$ by a scalar $c_\mu(\lambda)$. The problem of calculating this scalar 
(the \emph{Capelli eigenvalue problem})
has a long history (see below).
In this paper we complete the solution of this problem in the super setting.

For ordinary Jordan algebras (i.e., when $J_\ood=\{0\}$),  $\sP(V)$ is a multiplicity-free representation of the reductive Lie algebra $\g g$
(see \cite{Schmid}, \cite{JosephP}). In this case, 
the solution to the Capelli eigenvalue problem was given when $D_\lambda$ corresponds to a one-dimensional representation by  Kostant and the first author in \cite{KostantSahi91}, and later in full generality by the first author in \cite{Sahi94}.
Indeed in \cite{Sahi94} the first author introduced a \textquotedblleft
universal\textquotedblright\ family of symmetric polynomials $\varphi _{\mu
}^{\left( \rho \right) }( x) $ characterized by certain vanishing
properties, and depending on an auxiliary vector $\rho =\left( \rho
_{1},\ldots ,\rho _{n}\right) $. The main result of \cite{Sahi94} is that $%
c_{\mu }(\lambda )=$ $\varphi _{\mu }^{\left( r\delta \right) }(\lambda
+r\delta )$, such that $\delta: =\left( 0,-1,\ldots ,-n+1\right) $ 
where $n$ is the rank of the symmetric space 
$(\g g,\g k)$
associated to the Jordan algebra $J$,
and $r$ is
half the multiplicity of restricted roots. 

The polynomials $\varphi _{\mu }^{\left( r\delta \right) }( x) $
were studied by Knop and the first author  in 
\cite{KSDiff}
 for arbitrary $r$, who proved that they
satisfy a system of difference equations, which are a discrete version of
the Debiard-Sekiguchi system for Jack polynomials
\cite{Debiard}, \cite{Sekiguchi}.  Knop and the first author deduced that the
top-degree terms of  $\varphi _{\mu }^{\left( r\delta \right) }$ are
proportional to the Jack polynomials $P_{\mu }^{\left( 1/r\right) }$. For this
reason, the $\varphi _{\mu }^{\left( r\delta \right) }( x) $ are
sometimes referred to as Knop--Sahi polynomials, or shifted Jack
polynomials. The supersymmetric analogue of these polynomials was constructed
by Sergeev and Veselov \cite{SVSuperJack}. The top-degree terms of the Sergeev-Veselov polynomials $\SPP_\lambda(x,y,\theta)$ are the super Jack polynomials. An analogous family of polynomials $Q_\lambda^*(x)$ whose top-degree terms are the 
Schur $Q$-polynomials was defined by Okounkov and Ivanov
\cite{Ivanov}.

%Indeed the polynomials $\varphi_\mu$ are obtained from \emph{interpolation Jack polynomials} for a  value of the parameter $\theta$ that is associated to the restricted root system of the symmetric pair $(\g g,\g k)$, where $\g k$ is the stabilizer of the unit element of $J$.

The study of the Capelli eigenvalue problem 
for Jordan superalgebras was initiated in \cite{SahSalAdv}, where it was solved in the cases $J\cong \mathit{gl}(m,n)_+$ and $J\cong \mathit{osp}(n,2m)_+$.
These Jordan superalgebras correspond to  
symmetric pairs of types 
$(\gl\times \gl,\gl)$ and $(\gl,\g{osp})$, 
respectively. 
 Extending the results of Kostant and Sahi to these Jordan superalgebras,  in \cite{SahSalAdv} the first two authors showed that  the eigenvalues  of the Capelli operators are obtained by  specialization of the polynomials $\SPP_\mu$
at $\theta=1,\frac{1}{2}$. Later, 
in \cite{AllSahSal}
the Capelli eigenvalue problem was considered for Jordan superalgebras of type $\mathit{q}(n)_+$, and it was shown that the eigenvalues $c_\mu(\lambda)$ are given by the polynomials $Q_\mu^*$.

In this paper, we complete the project started in
\cite{SahSalAdv} and \cite{AllSahSal}, and  solve the Capelli eigenvalue problem for general unital simple Jordan superalgebras.  The new phenomenon that
arises in the present setting is  the occurrence of  certain deformations of the  root system of
the Lie superalgebra $\gl(r|s)$, 
studied by Sergeev and Veselov \cite{SVDQ}, which we define below. 

Let $r,s\geq 0$ be integers. 
We represent the roots of the root system $\mathsf{A}(r-1,s-1)$ by 
\begin{equation}
\label{eq:rootsofAr1s1}
\mathsf{R}_{r,s}:=\left\{
\udl\eps_i-\udl\eps_{i'}\right\}_{1\leq i\neq i'\leq r}
\cup
\left\{\udl\delta_j-\udl\delta_{j'}\right\}_{1\leq j\neq j'\leq s}
\cup\left\{\pm\left(\udl\eps_i-\udl\delta_{j}\right)\right\}_{1\leq i\leq r,1\leq j\leq s},
\end{equation}
as a subset of the $(r+s)$-dimensional vector space
$\mathsf{E}_{r,s}:=\spn_{\R}\left\{\udl\eps_i,\udl\delta_j\,:\,1\leq i\leq r,1\leq j\leq s\right\}$.
Fix $\kappa\in\R$ (if $s>0$, we assume $\kappa\neq 0$), and   let  $\lag\cdot,\cdot\rag_\kappa^{}$ be a (unique up to a scalar)  nondegenerate symmetric bilinear form on $\mathsf{E}_{r,s}$ such that $\left\{\udl\eps_i^{}\right\}_{i=1}^r
\cup
\left\{\udl\delta_j\right\}_{j=1}^s$ is an orthogonal basis of $\mathsf{E}_{r,s}$ with respect to 
$\lag\cdot,\cdot\rag_\kappa^{}$ that satisfies
\[
\lag\udl\eps_{i},\udl\eps_{i}\rag_\kappa^{}
=
\lag\udl\eps_{j},\udl\eps_{j}\rag_\kappa^{}
=
\kappa^{-1}
\lag\udl\delta_{i'},\udl\delta_{i'}\rag_\kappa^{}
=
\kappa^{-1}
\lag\udl\delta_{j'},\udl\delta_{j'}\rag_\kappa^{}
\ \text{ for }\
1\leq i,j\leq r\text{ and }
1\leq i',j'\leq s.
\]
The deformed root system $\mathsf{A}_\kappa(r-1,s-1)$ is the subset $\mathsf{R}_{r,s}$ of
the quadratic space $\left(\mathsf{E}_{r,s},\lag \cdot,\cdot\rag_\kappa\right)$. 
The root multiplicities of $\mathsf{A}_\kappa(r-1,s-1)$ are defined 
to be
\begin{equation}
\label{multiplicitiess}
\mathrm{mult}(\udl\eps_i-\udl\eps_{i'}):=\kappa,\ 
\mathrm{mult}(\udl\delta_j-\udl\delta_{j'}):=
\kappa^{-1},
\text{ and }
\mathrm{mult}(\udl\eps_i-\udl\delta_j):=1.
\end{equation}

For convenience, from now on we assume that $J_\eev\neq\{0\}$. 
We remark that our techniques and results can easily be adapted to ordinary Jordan algebras, and the reason for excluding them is that they have been dealt with in \cite{Sahi94}.

\begin{rmk}If $J\cong \mathit{JP}(0,n)$, then $\sP^2(V)$ is not completely reducible. Therefore without loss of generality, from now on we  exclude the Jordan superalgebras 
$\mathit{JP}(0,n)$.
\end{rmk}

Our next goal is to associate a set of restricted roots $\mathit{\Sigma}$ to $J$.  
The Lie superalgebra $\g g$ that is associated to $J$ is isomorphic to  one of the types $\gl$, $\gl\times \gl$, $\g{gosp}$, or $\g q\times \g q$.
Throughout the paper, we will use a standard  matrix realization of  $\g g$ 
that is given in Section \ref{sec-BorelKemd}. In this realization, 
there is a natural Cartan subalgebra $\g h\sseq \g g$ such that $\g h_\eev$ is equal to the subspace of diagonal matrices and $\Theta(\g h)=\g h$. 
Note that $\g h_\ood=\{0\}$ except when $J\cong\mathit{q}(n)_+$ for $n\geq 2$. 
%The Tits-Kantor-Koecher construction yields a symmetric subalgebra $\g k$ of $\g g$, which is the set of fixed points of the involution $\Theta:\g g\to\g g$, $\Theta(x):=\Ad_wx$, where $w$ is the nontrivial element of the Weyl group of the short subalgebra defined in \eqref{Eqhefrel} (see Appendix \ref{Sec-TKKconst}).

Since  $\Theta(\g h_\eev)=\g h_\eev$, we have 
a direct sum decomposition \[
\g h_\eev=\g t_\eev\oplus\g a_\eev,
\] where 
$\g t_\eev$ and $\g a_\eev$ are the $+1$ and $-1$ eigenspaces of $\Theta\big|_{\g h_\eev}$, respectively.
Let $\mathit{\Delta}$ denote the root system of $\g g$  corresponding to $\g h_\eev$, and  
set 
\begin{equation}
\label{eq:DfnofitSigm}
\mathit{\Sigma}:=
\left\{\alpha\big|_{\g a_\eev}\,:\,\alpha\in\mathit{\Delta}\right\}\bls \{0\}.
\end{equation}
Assume that $\mathit{\Sigma}\neq\varnothing$ (see
Remark \ref{rmk:caseooffosp12n}). Then according to the structure of $\mathit{\Sigma}$,
the Jordan superalgebras $J$  can be divided into  two classes (type $\mathsf{A}$ and type $\mathsf{Q}$)  defined below. 
\subsection*{Jordan superalgebras of type $\mathsf{A}$} 
Assume that  $J$ is one of the Jordan superalgebras that appear in  Table \ref{tbl-JtypeA}. Then 
$\mathit{\Sigma}$ is a root system of type $\mathsf{A}(r-1,s-1)$, where 
$r:=\sfr_{J,+}$ and $s:=\sfr_{J,-}$ are given in Table 
\ref{tbl-JtypeA}. 
We represent  this root system as in 
\eqref{eq:rootsofAr1s1}. 
Furthermore, in these cases $\gflat$ always has an invariant non-degenerate supersymmetric even bilinear form (see Table \ref{tbl-1}). Fix such a bilinear form $\lag\cdot,\cdot\rag_{\flat}$ on $\gflat$ (the choice of the bilinear form will not matter in what follows). Then the restriction
$\lag\cdot,\cdot\rag_{\flat}\big|_{\g a_\eev\times\g a_\eev}$ is also non-degenerate,  and therefore it induces an isomorphism $\g a_\eev^{}\cong \g a^*_\eev$. Via the latter isomorphism, $\lag\cdot,\cdot\rag_{\flat}\big|_{\g a_\eev\times\g a_\eev}$ induces a bilinear form \[
\lag\cdot,\cdot\rag_J^{}:\g a^*_\eev\times \g a^*_\eev\to\C.
\] 
For $\alpha\in\mathit{\Sigma}$, we denote the corresponding restricted root space of $\g g$ by  $\g g_\alpha$. We define the multiplicity of 
each $\alpha\in\mathit{\Sigma}$ to be 
\[
\mathrm{mult}
(\alpha):=-\frac{1}{2}\mathrm{sdim}(\g g_\alpha),
\]
where 
 for any $\Z/2$-graded vector space $E:=E_\eev\oplus E_\ood$ we define
$\mathrm{sdim} E:=\dim E_\eev-\dim E_\ood$.
The last 3 columns of Table \ref{tbl-JtypeA} give 
the graded dimensions of the restricted root spaces.

One can now verify directly that $\mathit{\Sigma}$, 
considered as a subset of the quadratic space $\big(\g a^*_\eev,\lag\cdot,\cdot\rag_J^{}\big)$ and equipped with the multiplicities defined above, is the
deformed root system  $\mathsf{A}_{\kappa}(r-1,s-1)$. 
Set \[
\theta_J:=-\kappa.
\] Thus, the value of $\theta_J$ can be obtained from 
either of the two equalities
\begin{equation*}
%\label{eq:forTHEJJ}
\theta_J=
-\frac{\lag \udl\delta_1,\udl\delta_1\rag_J^{}}
{\lag \udl\eps_1,\udl\eps_1\rag_J^{}}
\,\text{ and }\,
\theta_J=
\frac{1}{2}\mathrm{sdim}(\underline{\eps}_i-\underline{\eps}_j),
\end{equation*}
and indeed in the cases that both of the quantities
$-\frac{\lag \udl\delta_1,\udl\delta_1\rag_J^{}}
{\lag \udl\eps_1,\udl\eps_1\rag_J^{}}$
and $
\frac{1}{2}\mathrm{sdim}(\underline{\eps}_i-\underline{\eps}_j)
$ are well-defined, they are equal.
The values of $\theta_J$ are given in Table \ref{tbl-JtypeA}. The details of the computations that yield the values of the parameters 
$\sfr_{J,+}$,  
$\sfr_{J,-}$,
and $\theta_J$ are  postponed until Section 
\ref{sec-BorelKemd}.

\begin{rmk}
In Case IV of Table \ref{tbl-JtypeA},
 we assume that $t\in\C\bls\{0,-1\}$ because 
$\mathit{D}_0$ is not simple and $\mathit{D}_{-1}\cong \mathit{gl}(1,1)_+$.
\end{rmk}

%In all of the cases except Case III of Table \ref{tbl-JtypeA} we have $r,s>0$, and In Case III we have $s=0$ and threfore one cannot use formula \eqref{eq:forTHEJJ}, but one can still consider $\mathit{\Sigma}$ as such a deformed root system, where the value of $\theta_J$ is obtained from the graded multiplicity of $\eps_1-\eps_2$ using \eqref{multiplicitiess}. 

\begin{table}[h]

\begin{tabular}{ccccccccc}
 & 
$J$ & 
Remarks&
$\sfr_{J,+}$& 
$\sfr_{J,-}$&
$\theta_J$ &
${\pm(\underline{\eps}_i-\underline{\eps}_j)}$&
${\pm(\underline{\eps}_i-\underline{\delta}_j)}$&
${\pm(\underline{\delta}_i-\underline{\delta}_j)}$ 
\\
\hline 
I&
$\mathit{gl}(m,n)_+$ &
$m,n\geq 1$&
$m$ &
$n$ & 
$1$&
$2|0$ & 
$0|2$ & 
$2|0$ \\

II &
$\mathit{osp}(n,2m)_+$& 
$m,n\geq 1$ &
$m$ & 
$n$ &
$\frac{1}{2}$&
$1|0$ &
$0|2$ &
$4|0$ \\

III&
$(m,2n)_+$&
$m,n\geq 1$&
$2$&
$0$&
$\frac{m-1}{2}-n$ &
$m-1|2n$ &
$-$& 
$-$\\

IV&
$\mathit{D}_t$&
$t\neq 0,-1$&
$1$ &
$1$ &
$-\frac{1}{t}$&
$-$ & 
$0|2$ &
$-$  \\

V&
$\mathit{F}$ &
&
$2$&
$1$&
$\frac{3}{2}$&
$3|0$&
$0|2$&
$-$\\

\hline

\end{tabular}

\smallskip

\smallskip

\smallskip

\smallskip

\caption{$\mathit{\Sigma}$ of Type $\mathsf{A}$.}
\label{tbl-JtypeA}
\end{table}

\subsection*{Jordan superalgebras of type $\mathsf{Q}$}
Next assume that $J$ is one of the Jordan superalgebras that appear in Table \ref{tbl-JtypeQ}.
Then $\mathit{\Sigma}$ is a root system of type $\mathsf{Q}(r)$, where $r:=r_J$ is given in 
Table \ref{tbl-JtypeQ}. The graded dimension of all of the 
restricted root spaces is $(2|2)$.

\begin{table}[h]

\begin{tabular}{cccc}
&
$J$ & 
Remarks &
$\sfr_J$  \\

\hline
VI &
$\mathit{p}(n)_+$& 
$n\geq 2$ &
$n$\\

VII &
$\mathit{q}(n)_+$& 
$n\geq 2$ &
$n$\\

\hline

\end{tabular}

\smallskip

\smallskip

\smallskip

\smallskip

\caption{$\mathit{\Sigma}$ of Type $\mathsf{Q}$.}
\label{tbl-JtypeQ}
\end{table}

In the following remark,  $(m,2n)_+$ denotes the Jordan superalgebra with underlying space $\C1\oplus E$ and 
with product $a\circ b:=(a,b)_E1$, where $E$ is an $(m|2n)$-dimensional vector superspace equipped with a nondegenerate even supersymmetric bilinear form 
$(\cdot,\cdot)_{E}$. 

\begin{rmk}
\label{rmk:caseooffosp12n}
The only cases for which
$\mathit{\Sigma}=\varnothing$ are the Jordan superalgebras
 of type $(0,2n)_+$. Indeed it appears 
that 
the situation for these Jordan superalgebras
differs substantially from the other cases that are considered in this paper, for the following reasons.
First,  the Zariski closure 
of the set of highest weights 
that occur in $\sP(V)$
is not an affine subspace (see
Definition 
\ref{dfnaOmeg*} and Remark \ref{rmk-technicalasmptt}), and therefore it does not seem to be natural to consider the eigenvalues of the $D_\mu$  as a polynomial function on this Zariski closure (see Theorem \ref{thm-main-opVla}).
Second, even though $\sP(V)$ is a completely reducible and multiplicity-free $\g g$-module
(see \cite[Sec. 5.3]{CoulembierJLie2013}), the highest weights that occur in $\sP(V)$ look quite different from those that occur in the cases $J\cong (m,2n)_+$ for $m>0$. In particular,  the number of irreducible $\g g$-submodules occurring in the subspace $\sP^k(V)$ of homogeneous elements of degree $k$ in $\sP(V)$ stabilizes for $k\geq 2n$. Therefore unlike
the case $J\cong(m,2n)_+$ for $m>0$,  one cannot expect a parametrization of irreducible summands of $\sP(V)$ by 
hook partitions 
(see \eqref{mnhookPPP} below). 
 We hope to investigate these interesting cases in the future. 
In the rest of this paper, we assume that $J\not\cong(0,2n)_+$.
\end{rmk}

In order to state our first theorem (Theorem \ref{thm:compl-red}), we need  the parametrization of the irreducible summands of $\sP(V)$ by partitions.
For this parametrization,  we 
choose a Borel subalgebra \[
\g b:=\g h\oplus\g n
\] 
 of $\g g$ satisfying $\g g=\g k+\g b$. For the precise definition of $\g b$ 
and the embedding of $\g k$ as a subalgebra of $\g g$,
 see Section \ref{sec-BorelKemd}.
The quintuples $(\gflat,\g g,\g k,\g b,V)$
that are associated to the Jordan superalgebras $J$ are also listed in Table \ref{tbl-1}. 
%We set $\mathit{\Sigma}^\pm:=\mathit{\Delta}^\pm$, where  $\mathit{\Delta}^+$ is the set of positive roots of $\g g$ corresponding to $\g b$, and $\mathit{\Delta}^-=-\mathit{\Delta}^+$.

 Let $\EuScript{P}$ denote the set of partitions. We represent elements of $\EuScript{P}$ by sequences of integers $\lambda:=(\lambda_i)_{i=1}^\infty$ such that $\lambda_i\geq \lambda_{i+1}$ for all $i\geq 1$, and $\lambda_i=0$ for all sufficiently large $i\in\N$.
As usual, the \emph{weight} of any $\lambda\in\EuScript{P}$ is defined by 
$|\lambda|:=\sum_{i=1}^\infty\lambda_i$.
A partition $\lambda:=(\lambda_i)_{i=1}^\infty\in\EuScript{P}$ is called \emph{strict} if 
$\lambda_i>\lambda_{i+1}$ for all $i\leq\ell(\lambda)$, where 
$\ell(\lambda):=\max\{i\,:\,\lambda_i>0\}$ denotes the length of $\lambda$.
For $n\geq 0$ let $\EuScript{DP}(n)$ be the set of strict partitions 
$\lambda$ such that $\ell(\lambda)\leq n$.
For $m,n\geq 0$, let 
$\EuScript{H}(m,n)$ be the set of \emph{$(m,n)$-hook} partitions, defined by 
\begin{equation}
\label{mnhookPPP}
\EuScript{H}(m,n):=\left\{ 
\lambda\in\EuScript{P}\,:\,\lambda_{m+1}\leq n
\right\}.
\end{equation}
For $d\geq 0$ set 
 \[
\EuScript{H}_d(m,n):=\left\{\lambda\in\EuScript{H}(m,n)\,:\,
|\lambda|=d\right\}\text{ and } \EuScript{DP}_d(n):=\{\lambda\in\EuScript{DP}(n)\,:\,|\lambda|=d\}.
\]
Also, set
\[
\mathcal S(m,n):=
\begin{cases}
\left\{-\frac{a}{b}\,:\,a,b\in\Z,\ a\geq 1,\ \text{and}\ 1\leq b\leq m-1\right\}&\text{ if }n=0,\\
\left\{-\frac{a}{b}\,:\,a,b\in\Z,\ 0\leq a\leq n,\ \text{and}\ b\geq 1\right\}&\text{ if }m=0,\\
\Q^{\leq 0}&\text{ otherwise.}
\end{cases}
\] 
Indeed $\mathcal S(m,n)$ is the set of admissible parameter values of the Sergeev-Veselov polynomials (see Theorem \ref{thm:SVSP*}). 
The first main result of this paper is the following.
\begin{thm}
\label{thm:compl-red}
Let $J$ be a finite dimensional unital complex simple Jordan superalgebra such that $J_\ood\neq\{0\}$.
Further, assume that $J$ is not isomorphic to one of the Jordan superalgebras of types $(0,2n)_+$ and $ \mathit{JP}(0,n)$.
Let  $\g g$, $\g b$, and $V$ be associated to $J$ as above.  Then the following assertions hold. 
\begin{itemize}
\item[\upshape(i)] When $J$ is of type 
$\mathsf{A}$,  the  $\g g$-module $\sP(V)$ is  completely reducible and multiplicity-free if and only if 
$\theta_J\not\in\mathcal S(\sfr_{J,+},\sfr_{J,-})$.
\item[\upshape(ii)] When $J$ is of type 
$\mathsf{Q}$, the $\g g$-module $\sP(V)$ is  completely reducible and multiplicity-free. 
\end{itemize}
Furthermore, whenever $\sP(V)$ is 
completely reducible and 
multiplicity-free,  for every $d\geq 0$ we have  
\begin{equation}
\label{eqthm1.9sPD}
\sP^d(V)\cong\bigoplus_{\lambda\in\Omega_d}V_\lambda
,
\end{equation} where 
$V_\lambda$ is the 
irreducible $\g g$-module
with the $\g b$-highest weight $\hww{\lambda}$ 
 given in Table \ref{tbl-3}, and 
\[
\Omega_d:=\begin{cases}
\EuScript{H}_d(\sfr_{J,+},\sfr_{J,-})
&\text{ if $J$ is of type $\mathsf{A}$,}\\
\EuScript{DP}_d(\sfr_J^{})
&\text{ if $J$ is of type $\mathsf{Q}$.}\\
\end{cases}
\]
\end{thm}

\begin{rmk}
In Table \ref{tbl-3}, we represent the $\g b$-highest weight $\hww{\lambda}$
as a linear combination of the standard characters  of $\g h_\eev$,
when $\g g$ is 
realized as in  Section
\ref{sec-BorelKemd}. The standard character of the Lie superalgebras of types $\gl$, $\g{gosp}$, and $\g{q}$ are  given in 
 Appendices
\ref{subsec-gl}, \ref{subsec-osp}, and \ref{subsec-q} respectively.
The highest weights $\lambda_{m|n}^{\mathrm{st}}$
and $\lambda_n^{\mathrm{st}}$ 
that appear in Cases I and VII of Table \ref{tbl-3}
are 
defined in Appendices \ref{subsec-gl}.
and \ref{subsec-q} respectively.
\end{rmk}

\begin{rmk}
\label{rmk-someprvd}
We remark that several of the cases of 
Theorem \ref{thm:compl-red} 
 are already known. For $J$ corresponding to Cases I--III and VII of Table \ref{tbl-JtypeA} and Table \ref{tbl-JtypeQ},
Theorem \ref{thm:compl-red}  can be found in
\cite{Brinietal},
\cite{ChengWangLMP},
\cite{ChWacompositio}, 
and \cite{CoulembierJLie2013}. Thus,
the new cases of
Theorem \ref{thm:compl-red} are 
Cases IV--VI, for which the assertion
 is proved in Section \ref{sec-mulfree}.
\end{rmk}

In the rest of this section we assume that 
the $\g g$-module $\sP(V)$ is 
completely reducible and multiplicity-free. Set
\[
\Omega:=\bigcup_{d\geq 0}\Omega_d
.
\]
Then from \eqref{eqthm1.9sPD}  it follows that 
\begin{equation}
\label{eq:capisodfnn}
\sPD(V)^{\g g}\cong \left(\sP(V)\otimes\sS(V)\right)^{\g g}\cong \bigoplus_{\lambda,\mu\in\Omega}
\left(V_\lambda^{}\otimes V_\mu^*\right)^{\g g}\cong
\bigoplus_{\lambda,\mu\in\Omega}
\Hom_\g g(V_\mu,V_\lambda).
\end{equation}
For $\lambda\in\Omega$, let $D_\lambda$ be the element of $\sPD(V)^\g g$ that corresponds  to  $\mathrm{id}_{V_\lambda}\in \Hom_\g g(V_\lambda,V_\lambda)$
via the  isomorphism \eqref{eq:capisodfnn}.
\begin{rmk}
\label{rem:CapelliEven}
The family $(D_\lambda)_{\lambda\in\Omega}$ forms a basis of $\sPD(V)^{\g g}$. This is because
$V_\lambda$ is of type $\mathsf M$ in the sense of \cite[Sec. 3.1.2]{CWBook}, that is, $V_\lambda$ is irreducible as an ungraded module. In particular, there is no odd $\g g$-intertwining map $V_\lambda\to V_\lambda$. 
For Cases I--VI, this property of $V_\lambda$ is an immediate consequence of  highest weight theory for  Lie superalgebras of types $\gl$ and $\g{osp}$,
and for Case VII, it is verified in 
\cite[Sec. 3.1]{AllSahSal}.
\end{rmk}

Our second main result (Theorem \ref{thm-main-opVla})
yields an explicit formula for the eigenvalue $c_\mu(\lambda)$ of $D_\mu$ on $V_\lambda$. Before we state
Theorem \ref{thm-main-opVla}, we  need to recall the definitions of the shifted super Jack polynomials of Sergeev and Veselov \cite{SVSuperJack}, and the factorial Schur $Q$-polynomials of 
Okounkov and Ivanov \cite{Ivanov}.

For $m,n\geq 0$ let
$\sP_{m,n}$ denote the $\C$-algebra of 
polynomials in $m+n$ variables 
$x_1,\ldots,x_m$ and $y_1,\ldots,y_{n}$.
Fix $\theta\in\C$ 
(if $n>0$, we assume $\theta\neq 0$). 
Let 
%$\theta$ be a complex number outside the set of poles of the  coefficients of the polynomials $\SPP_\lambda(x,y,\vartheta)$.By setting $\vartheta:=\theta$,  from each $\SPP_\lambda(x,y,\vartheta)$ we obtain an element of the algebra 
$\Lambda_{m,n,\theta}^\natural\sseq\sP_{m,n}$ be the 
subalgebra of polynomials $f(x,y)$ with complex coefficients 
which are separately symmetric in 
$x:=(x_1,\ldots,x_m)$ and in $y:=(y_1,\ldots,y_n)$, 
and which satisfy the relation
\begin{equation*}
%\label{eq:fx12+e-ei}
\textstyle f\left(x+\frac{1}{2}\sfe_i,y-\frac{1}{2}\sfe_j\right)=
f\left(x-\frac{1}{2}\sfe_i,y+\frac{1}{2}\sfe_j\right)
\end{equation*}
on every hyperplane $x_i+\theta y_j=0$, where $1\leq i\leq m$ and $1\leq j\leq n$.
Given any $\lambda\in\EuScript{H}(m,n)$, as in   \cite[Sec. 6]{SVSuperJack} 
for $1\leq i\leq m$ and $1\leq j\leq n$
we define
%%%%%%%%%%%%%%%
\begin{equation}
\label{eq:piqiFrb}
\textstyle
\sfp_i(\lambda):=\lambda_i-\theta\left(i-\frac{1}{2}\right)
-\frac{1}{2}(n-\theta m)\ \text{ and }\ 
\sfq_j(\lambda):=\lag \lambda'_j-m\rag-\theta^{-1}\left(j-\frac{1}{2}\right)+\frac{1}{2}\left(\theta^{-1}n+m\right),
\end{equation}
where $\lambda'$ denotes the transpose of $\lambda$,
and 
\[
\lag x\rag :=\max\{x,0\}\text{ for }x\in\R.
\] 
The $(m+n)$-tuple $(\sfp(\lambda),\sfq(\lambda))$, where $\sfp(\lambda):=(\sfp_1(\lambda),\ldots,\sfp_m(\lambda))$ and $\sfq(\lambda):=(\sfq_1(\lambda),\ldots,\sfq_n(\lambda))$, is called the \emph{Frobenius coordinates} of $\lambda$.
The following theorem characterizes shifted super Jack polynomials by
their degree, symmetry, and vanishing properties.
\begin{thm}
\label{thm:SVSP*}
{\rm (Sergeev--Veselov \cite[Thm 3]{SVSuperJack}, Knop--Sahi \cite[Sec. 2]{KSDiff})}
Let $m,n\geq 0$ be integers and let $\theta$ be a complex number such that  $\theta\not\in \mathcal S(m,n)$. Then for each $\lambda\in\EuScript{H}(m,n)$, there exists a unique polynomial $\SPP_\lambda\in\Lambda_{m,n,\theta}^\natural$ that satisfies the following properties.
\begin{itemize}
\item[\upshape (i)] $\deg(\SPP_\lambda)\leq |\lambda|$, where $\deg(\SPP_\lambda)$ denotes the total degree 
of $\SPP_\lambda$
in $x$ and $y$.
\item[\upshape(ii)] $\SPP_\lambda(\sfp(\mu),\sfq(\mu),\theta)=0$ for all $\mu\in\EuScript{H}(m,n)$ such that $|\mu|\leq |\lambda| $ and $\mu\neq \lambda$.
\item[\upshape(iii)]  
$\SPP_\lambda(\sfp(\lambda),\sfq(\lambda),\theta)=H_\theta(\lambda)$,
where
\[H_\theta(\lambda):=
\prod_{1\leq i\leq \ell(\lambda)}
\prod_{1\leq j\leq \lambda_i}
(\lambda_i-j+\theta(\lambda'_j-i)+1).
\]
\end{itemize} 
Furthermore, the family of polynomials 
$
\big(\SPP_\lambda(x,y,\theta)\big)_{\lambda\in\EuScript{H}(m,n)}
$ is a basis of $\Lambda^\natural_{m,n,\theta}$. 
\end{thm}

\begin{rmk}
\label{rmk:m=2n=0}
For $m,n>0$, Theorem \ref{thm:SVSP*} is proved in
\cite{SVSuperJack}.
If either $m=0$ or $n=0$,  
then up to scaling the  $\SPP_\lambda$ 
are the same as the interpolation polynomials $P_\lambda^\rho$ defined by Knop and Sahi  in
\cite{KSDiff}. 
%The latter polynomials were first introduced in \cite{Sahi94}. 
Given $k\in\N$ and $\alpha\in\C$, set  $\rho_{k,\alpha}^{}:=(\rho_1,\ldots,\rho_k)$
where $\rho_i:=\frac{\alpha}{2}(k-2i+1)$ for $1\leq i\leq k$. 
If $n=0$ then  $\SPP_\lambda(x,\theta)=P_\lambda^\rho$
for
$\rho:=\rho_{m,\theta}^{}$, and if
$m=0$ then $\SPP_\lambda(y,\theta)=
\frac{H_\theta(\lambda)}{H_{\theta^{-1}}(\lambda')}P_{\lambda'}^{\rho}$
for $\rho:=\rho_{n,\theta^{-1}}^{}$.
\end{rmk}

Next we state the characterization of factorial Schur $Q$-polynomials by their degree, symmetry and vanishing properties. For $n\in\N$, let $\sP_n$ denote the $\C$-algebra of  polynomials in $n$ variables $x_1,\ldots,x_n$.
Further, let  $\Gamma_n\sseq \sP_n$ be the subalgebra of symmetric polynomials $f(x_1,\ldots,x_n)$ such that 
$f(t,-t,x_3,\ldots,x_n)$ is independent of $t$ (for $n=1$ the latter condition is vacuous). 
In Theorem
\ref{thm-Iv}, 
for  $\lambda\in\EuScript{DP}(n)$ 
we define $\lambda!:=\prod_{1\leq i\leq \ell(\lambda)}\lambda_i!$ and identify $\lambda$ with the $n$-tuple $(\lambda_1,\ldots,\lambda_n)\in\C^n$.

\begin{thm}
\label{thm-Iv}
{\upshape (Ivanov \cite[Sec. 1]{Ivanov})}
For every $\lambda\in\EuScript{DP}(n)$, there exists a unique polynomial $Q_\lambda^*\in\Gamma_n$ which satisfies the following properties.
\begin{itemize}
\item[{\upshape (i)}] $\deg(Q_\lambda^*)\leq |\lambda|$.
\item[{\upshape (ii)}] $Q_\lambda^*(\mu)=0$ for all $\mu\in\EuScript{DP}(n)$ such that $|\mu|\leq |\lambda|$
and $\mu\neq \lambda$.
\item[{\upshape (iii)}] $Q_\lambda^*(\lambda)=H(\lambda)$,
where
$H(\lambda):=\lambda!\prod_{1\leq i<j\leq \ell(\lambda)}\frac{\lambda_i+\lambda_j}{\lambda_i-\lambda_j}$.

\end{itemize}
Furthermore, the family of polynomials $\big(Q_\lambda^*\big)_{\lambda\in\EuScript{DP}(n)}$ is a basis of $\Gamma_n$.
\end{thm}

Set 
\begin{equation}
\label{eq:LamJPJ}
\Lambda_J:=\begin{cases}
\Lambda^\natural_{\sfr_{J,+}^{},\sfr_{J,-}^{},\theta_J}
& \text{ if $J$ is of type $\mathsf{A}$,}\\
\Gamma_{\sfr_J}
&\text{ if $J$ is of type $\mathsf{Q}$.}
\end{cases}
\quad\text{ and }\quad
\sP_J:=\begin{cases}
\sP_{\sfr_{J,+}^{},\sfr_{J,-}^{}}^{}
&\text{ if $J$ is of type $\mathsf{A}$,}\\
\sP_{\sfr_J}
&\text{ if $J$ is of type $\mathsf{Q}$.}
\end{cases}
\end{equation}
There is a natural embedding of $\Lambda_J$ as a  subalgebra of $\sP_J$.

\begin{dfn}
\label{dfn:Omega}
For $\lambda\in\Omega$, we define $P_{J,\lambda}\in\Lambda_J$ as follows. When $J$ is of type $\mathsf{A}$
we set \[
P_{J,\lambda}:=\frac{|\lambda|!}{H_{\theta_J}(\lambda)}\SPP_\lambda(x,y,\theta_J),
\] where 
$x:=\left(x_1,\ldots,x_{\sfr_{J,+}}\right)$ and $y:=\left(y_1,\ldots,y_{\sfr_{J,-}}\right)$. When $J$ is of type $\mathsf{Q}$ we set 
\[
P_{J,\lambda}:=\frac{|\lambda|!}{H(\lambda)}Q^*_\lambda(x),
\] where $x:=\left(x_1,\ldots,x_{r_J}\right)$. 
\end{dfn}

%Set $\EuScript{P}(k):=\{\lambda\in \EuScript{P}\,:\,|\lambda|=k\}$ for  $d\geq 0$.

Recall that for $\lambda\in\Omega$, we denote  the $\g b$-highest weight of the irreducible $\g g$-module $V_\lambda$ by $\hww{\lambda}$.

\begin{dfn}
\label{dfnaOmeg*} We define $\aomega$ to be the Zariski closure of the set 
$\big\{\hww{\lambda}\,:\,\lambda\in\Omega\big\}$ in $\g h_\eev^*$. 
\end{dfn}
By a straightforward calculation using the  explicit description of the $\g b$-highest weights
given in Table \ref{tbl-3}, one can verify that 
$\aomega$ is a linear subspace of $\g h_\eev^*$ (see the proof of Proposition \ref{prp-JresHCZiLJ}
).
Set
\[
n_J:=\begin{cases}
\sfr_{J,+}+\sfr_{J,-}&\text{ if $J$ is of type $\mathsf{A}$,}\\
\sfr_J&\text{ if $J$ is of type $\mathsf{Q}$.}
\end{cases}
\]
Let
\begin{equation}
\label{eq:introtauJ}
\tau_J:\aomega\to \C^{n_J}
\end{equation}
be the affine linear map given in Table \ref{tbl-3},
where
the elements of $\aomega$ are given 
in Cases I--VII by 
\eqref{eq:muI}, \eqref{eq:muII},
\eqref{eq:muIII}, \eqref{eq:muIV},
\eqref{eq:muabc}, 
\eqref{eq:muVI},
and \eqref{eq:muVII} respectively,
and
the standard basis of $\C^{n_J}$ is denoted by $\sfe_1,\ldots,\sfe_{n_J}$.
We identify $\sP_J$ with the algebra of polynomials on 
$\C^{n_J}$ in the natural way. Namely,  for $v:=(v_1,\ldots,v_{n_J})\in\C^{n_J}$,
we identify the $x_i\in\sP_J$
with the maps $v\mapsto v_i$ 
 and the $y_j$ (if they exist) with the maps $v\mapsto v_{j+r_{J,+}}$.
 The second main result of this paper is the following.
\begin{thm}
\label{thm-main-opVla}
Let $J$, $\g g$, $\g b$, and $V$ be as in Theorem 
\ref{thm:compl-red}.
Assume that $\sP(V)$ is a 
completely reducible and 
multiplicity-free $\g g$-module. Then
for every $\lambda,\mu\in\Omega$, the operator $D_\mu$ 
%of Definition\ref{dfn:D-lambda} 
acts on $V_\lambda$ by the scalar $P_{J,\mu}\left(\tau_J\left(\underline{\lambda}\right)\right)$, where 
$P_{J,\mu}$ is as in Definition \ref{dfn:Omega},
and
$\hww{\lambda}$ is the $\g b$-highest weight of $V_\lambda$, given in Table \ref{tbl-3}. 
\end{thm}

\begin{rmk}
\label{someprvdII}
  Theorem 
\ref{thm-main-opVla} is proved  in \cite{SahSalAdv} 
for Cases I--II,
and in \cite{AllSahSal} 
for Case VII. 
We give a uniform proof for 
Cases I--IV and  and VI in
Sections 
\ref{Sec:surjectivity} and 
\ref{sec-thm-main-opVla} (see
Proposition \ref{prp::-surj-center}(iii)).
With minor modifications, this strategy also works for 
Case VII. However, 
 this uniform proof strategy does not work in Case V. 
In the latter case we prove Theorem 
\ref{thm-main-opVla} in Section 
\ref{Sec:10dimcase} by a different method. 
\end{rmk}

\begin{rmk}
\label{rmk-technicalasmptt}
%%%%%%%%%%%%%%%%%%%%%%%%%%%%%%%%%%
%If $J\cong \mathit{JP}(0,n)$, then 
%$\gJJ\cong H(n+3)$ and $\g g$ is isomorphic to a semidirect sum  $\widetilde{H}(n+1)\ltimes (\Lambda\C^{n+1}/\Lambda^{n+1}\C^{n+1})$, hence $\g g_\eev$ is not reductive.
%%%%%%%%%%%%%%%%%%%%%%%%%%%%%%%%%%
If $J\cong (0,2n)_+$, then $\sP(V)$ is completely reducible and multiplicity-free, but the Zariski closure of the set of highest weights is a union of $n+1$ lines. Therefore it is not possible to give a natural formulation of Theorem \ref{thm-main-opVla}.
\end{rmk}
\begin{rmk}
In Case VII, the operators $D_\lambda$ are closely related to certain operators $I_\lambda$ that are constructed by Nazarov \cite[Eq. (4.7)]{Nazarov} using characters
of the Sergeev algebra. Nazarov also defined certain explicit
``Capelli'' elements
in the centre of the enveloping algebra of $\g q(n)$, and proved 
\cite[Cor. 4.6]{Nazarov}
that their images under the left action of $\g q(n)$ on $V$ 
are the 
$I_\lambda$. The precise connection between the $D_\lambda$ and Nazarov's operators is determined
in 
\cite[Prop. 3.6]{AllSahSal}.

%\[\bigcup_{i=1}^n\{x\zeta+y(\delta_1+\cdots+\delta_i)\,:\,x,y\in\C\}\]
\end{rmk}

\begin{table}[h]
\renewcommand{\arraystretch}{2}
\scriptsize 

\begin{tabular}{ll
>{\raggedleft\arraybackslash}p{7.5cm}
}
&$\hspace{1cm}\underline{\lambda}$ 
&$\tau_J\hspace{2cm}$\\

\hline

I 
&$(-\lambda^{\mathrm{st}}_{m|n})\oplus\lambda^{\mathrm{st}}_{m|n}$
&
$\mu_{a,b}\mapsto\sum_{i=1}^m
\left(a_i+\frac{m-2i+1-n}{2}\right)\sfe_i+
\sum_{j=1}^n
\left(b_j+\frac{m-2j+1+n}{2}\right)\sfe_{m+j}$
    \\[2mm]

II    
&$-\sum_{i=1}^m 2\lambda_i\eps_i-\sum_{j=1}^n\lag\lambda_j'-m\rag(\delta_{2j-1}+\delta_{2j})$
&
$\mu_{a,b}\mapsto \sum_{i=1}^m\left(-\frac{1}{2}a_i+\frac{m+1-2n-2i}{4}\right)
\sfe_i+\sum_{j=1}^n\left(-b_j+\frac{m+2+2n-4j}{2}\right)\sfe_{m+j}$
 \\[2mm]

III  
& $(\lambda_1-\lambda_2)\eps_1+(\lambda_1+\lambda_2)\zeta$
& $\mu_{a,b}\mapsto \frac{1}{2}(a+b+m-2n-1)\sfe_1+\frac{1}{2}(b-a)\sfe_2$  \\[2mm]

IV  
&
$\left(\left(\frac{3+t}{1+t}\right)|\lambda|-2\lambda_1\right)\eps_1+\left(\lambda_1-
\left(\frac{2+t}{1+t}\right)|\lambda|\right)(\delta_1+\delta_2)$
&
$
\mu_{a,b}\mapsto \left(
-\frac{2+t}{1+t}a-\frac{3+t}{1+t}b
-\frac{1}{2}
\right)\sfe_1+$
$
\left(
\frac{1}{1+t}a-\frac{3+t}{1+t}b
+\frac{5+t}{1+t}
\right)\sfe_2
$

 \\[2mm]

V &$\left(3|\lambda|-2\lambda_1-2\lambda_2
\right)\eps_1+
\left(
\lambda_1-\lambda_2
\right)(\delta_1+\delta_2)+|\lambda|\zeta$   
&
$
\mu_{a,b,c}\mapsto \left(
\frac{-a+2b+3c+1}{4}
\right)
\sfe_1
+\left(
\frac{-a-2b+3c-5}{4}
\right)
\sfe_2
+$
$\left(
\frac{a-c+2}{2}
\right)\sfe_3
$

 \\[2mm]

VI   
&$-\sum_{i=1}^n \lambda_i\eps_i-\sum_{j=1}^n\lambda_j\delta_j$  
&
$
\mu_a\mapsto-\sum_{i=1}^na_i\sfe_i
$

 \\[2mm]

VII  
&$(-\lambda_n^\mathrm{st})\oplus\lambda_n^\mathrm{st}$
 &$\mu_a\mapsto \sum_{i=1}^na_i\sfe_i$   \\[2mm]
\hline

\end{tabular}

\smallskip

\smallskip

\smallskip

\smallskip

\caption{The $\g b$-highest weights $\hww{\lambda}$ 
and the affine maps $\tau_J$}
\label{tbl-3}
\end{table}

The reason why the uniform proof of Theorem \ref{thm-main-opVla} fails for Case V is a property of the image of the Harish-Chandra homomorphism which is of independent interest. 
We denote the universal  enveloping algebra of $\g g$ by  $\bfU(\g g)$. 
The $\g g$-action on $V$ 
%gives rise to a Lie superalgebra homomorphism $\g g\to\sPD(V)$, and the latter map 
induces a homomorphism of associative superalgebras
\begin{equation}
\label{eq:dfsfj}
\sfj:\bfU(\g g)\to \sPD(V)
.\end{equation}
Let 
$\left(\bfU^{(i)}(\g g)\right)_{i\geq 0}$
denote the standard filtration of $\bfU(\g g)$.
Let $\bfZ(\g g)\sseq\bfU(\g g)$ be the centre of $\bfU(\g g)$, and set $\bfZ^{(i)}(\g g):=\bfZ(\g g)\cap \bfU^{(i)}(\g g)$ for $i\geq 0$. 
Let 
\begin{equation}
\label{eq:HHCC1}
\HC:\bfU(\g g)\to \sS(\g h)\cong\sP(\g h^*)
\end{equation} 
be the Harish-Chandra projection 
corresponding to 
the triangular decomposition
$
\g g=\g n^-\oplus\g h\oplus \g n
$, where
$\g n^-$ is the nilpotent subalgebra of $\g g$ opposite
to $\g n$. Thus for $D\in\bfU(\g g)$ we define 
\[
\HC(D):=D_{\g h},
\] where $D=D_{\g h}+D'$ is the unique way of expressing  $D$ as a sum of two elements   $D_\g h\in\bfU(\g h)\cong \sS(\g h)$ and 
$D'\in\left(\bfU(\g g)\g n+\g n^-\bfU(\g g)\right)$.
 Let  
\begin{equation}
\label{eqdfreS}
\mathsf{res}:\sP(\g h^*)\to \sP(\aomega)
\end{equation} denote the canonical restriction map,
and let 
\begin{equation}
\label{eq:introtauJ2}
\tau_J^*:\sP_J\to\sP(\aomega)
\end{equation}
be the pullback of the map $\tau_J$ defined
in \eqref{eq:introtauJ}, that is, $\tau_J^*(p):=p\circ\tau_J$ for $p\in\sP_J$.
We denote the degree filtration of 
the algebra $\sP_J$ defined in
\eqref{eq:LamJPJ}  by $\left(\sP_J^{(i)}\right)_{i\geq 0}$.
In the following theorem, which will be proved in
Section \ref{Sec:surjectivity}, we denote the exceptional $(6|4)$-dimensional Jordan superalgebra by $\mathit{F}$.

\begin{thm}
\label{thm:surjZg}
Let $J$ be as in Theorem \ref{thm:compl-red}. Assume that $\sP(V)$ is completely reducible and multiplicity-free.
If $J\not\cong \mathit{F}$, then  $\mathsf{res}\left(\HC\left(\bfZ^{(i)}(\g g)\right)\right)= \tau_J^*\left(\Lambda_J^{(i)}\right)\text{ for }i\geq 0$,
where $\Lambda_J^{(i)}:=\Lambda_J^{}\cap \sP^{(i)}_J$. If $J\cong\mathit{F}$, then 
$\mathsf{res}\left(\HC\left(\bfZ(\g g)\right)\right)\subsetneq \tau_J^*\left(\Lambda_J^{}\right)$.
\end{thm}
Since the map $\sfj:\bfU(\g g)\to \sPD(V)$ is $\g g$-equivariant, we have $\sfj(\bfZ(\g g))\sseq\sPD(V)^{\g 
g}$.

\begin{cor}
\label{CorCor}
Let $J$ be as in Theorem \ref{thm:compl-red}. Assume that $\sP(V)$ is completely reducible and multiplicity-free.
If $J\not\cong\mathit{F}$, then $\sfj(\bfZ(\g g))=\sPD(V)^\g g$. If $J\cong\mathit{F}$, then
$\sfj(\bfZ(\g g))\subsetneq\sPD(V)^\g g$.
\end{cor}

This phenomenon already occurs in the non-super case
\cite{HoweUmeda}, \cite{HelgasonAJM92}. 
Corollary \ref{CorCor} follows from Proposition \ref{prp::-surj-center}(ii)
and Proposition 
\ref{prpJ===Fsjscfd}.

\begin{rmk}
Weingart 
 \cite{weingart}
computes the eigenvalues of a basis of invariant operators corresponding to the action of $\gl(n)$ on $\Lambda(\sS^2(\C^n))$ and $\Lambda(\Lambda^2(\C^n))$.
These multiplicity-free representations arise naturally from the action of the even part of the Lie superalgebra $\g p(n)$ on its odd part. We will study similar actions in a forthcoming paper. \end{rmk}

\section{Realizations of $\g g$, $\g k$,  $\g b$, $V$, and $\mathit{\Sigma}$}

\label{sec-BorelKemd}

In this section
we describe explicit embeddings of 
 $\g b$ and $\g k$  in $\g g$.  We have the following three possibilities for $\g g$.
\begin{itemize}
\item[\rm (i)]  $\g g\cong\gl(r|s)$ or $\g g\cong\gl(r|s)\oplus\gl(r|s)$ for some $r,s\in\N$.

\item[\rm (ii)] $\g g\cong\g{gosp}(r|2s)$ for some $r,s\in\N$. 

\item[\rm (iii)] $\g g\cong \g{q}(r)\oplus\g q(r)$ for some $r\geq 2$.
\end{itemize} 
In each of the cases (i)--(iii) above,
we consider the standard matrix realization of  
$\g g$ (or its direct summands) as given in Appendices \ref{subsec-gl}, \ref{subsec-osp}, and \ref{subsec-q}, respectively. The embedding $\g g\into\gflat$ is determined uniquely by the semisimple element $h$ given below.  
We identify $\g k$ and $\g b$ as subalgebras of this realization of $\g g$. We also give an explicit description of  $\mathit{\Sigma}$. In what follows, 
$\mathrm{diag}(X_1,\ldots,X_n)$ denotes the block diagonal matrix formed by $X_1,\ldots,X_n$.\\

%Throughout this section, for a Lie superalgebra of type $\gl$ we denote the standard characters of the diagonal Cartan subalgebra  by  $\eps_i$'s and $\delta_j$'s (see \ref{subsec-gl}). For a Lie superalgebra of type $\g{gosp}$,  we denote the standard characters of the diagonal Cartan subalgebra  by $\eps_i$'s, $\delta_j$'s, and $\zeta$ (see Appendix \ref{subsec-osp}). We remark that if $h\in\gflat$ is the Cartan element of the short $\g{sl}_2$ subalgebra  defined as in \eqref{Eqhefrel}, then $\zeta(h)=2$. Finally, for a Lie superalgebra of type  $\g q$, we denote the standard characters of the diagonal toral subalgebra  by $\check\eps_i$'s (see Appendix \ref{subsec-q}). \\

\noindent\textbf{Case I.} 
The matrix realization of $\gflat\cong\gl(2m|2n)$ is as in 
Appendix \ref{subsec-gl}. We set  
\[
h:=\mathrm{diag}(-I_{m\times m},I_{m\times m},-I_{n\times n},I_{n\times n}).
\] 
The matrix realization of $\g g\cong
\gl(m|n)\oplus\gl(m|n)$ is as in Appendix \ref{subsec-gl}, and the embedding $\g g\into\gflat$ is given by
\[
\left(
\begin{bmatrix}
A& B\\
C& D
\end{bmatrix}
,
\begin{bmatrix}
A'& B'\\
C'& D'
\end{bmatrix}
\right)
\mapsto
\begin{bmatrix}
A & 0_{m\times m} & B & 0_{m\times n}\\
0_{m\times m} & A' & 0_{m\times n} & B'\\
C& 0_{n\times m} & D & 0_{n\times n} \\
0_{n\times m} & C' & 0_{n\times n} & D'
\end{bmatrix}
.\]
Invariant 
supersymmetric
even bilinear forms on $\gflat$ are of the form 
$\lag x,y\rag:=\alpha_1\mathrm{str}(xy)+\alpha_2\mathrm{str}(x)\mathrm{str}(y)$ for  $\alpha_1,\alpha_2\in\C$. 
The realization of $\g a_\eev$ as a subalgebra of $\gflat$ is
\[
\g a_\eev:=\left\{\mathrm{diag}(\mathbf d_1,-\mathbf d_1,\mathbf d_2,-\mathbf d_2)\,:\,\mathbf d_1\in\C^m\text{ and }\mathbf d_2\in\C^n\right\}.
\]
Note that the bilinear map $(x,y)\mapsto \mathrm{str}(x)\mathrm{str}(y)$ vanishes on $\g a_\eev$, and therefore without loss of generality we can choose the invariant form $\lag x,y\rag_\flat$ to be
$\lag x,y\rag_\flat:=\mathrm{str}(xy)$.
Then $\mathit{\Sigma}$ is a root system of type $\mathsf{A}(m-1,n-1)$
for the choice of
$\underline{\eps}_i:=\eps_i\big|_{\g a_\eev}$ and $\underline\delta_j:=\delta_j\big|_{\g a_\eev}$,
 where
$\{\eps_i\}_{i=1}^m\cup\{\delta_j\}_{j=1}^n$ are the standard characters of the Cartan subalgebra $\g h_{m|n}$ of the left $\gl(m|n)$ summand of $\g g$, defined in \eqref{eq:CSAhm|n}. 
By a direct calculation we obtain $\lag \underline\eps_i,\underline\eps_j\rag_J^{}=\frac{1}{2}\delta_{i,j}$ and 
$\lag \underline\delta_i,\underline\delta_j\rag_J^{}=-\frac{1}{2}\delta_{i,j}$, so that $\theta_J=1$.

The embedding of $\g k\cong\gl(m|n)$ in $\g g$ is  the diagonal map $x\mapsto x\oplus x$. 
We set $\g b:=\g b^{\mathrm{op}}_{m|n}\oplus\g b^{\mathrm{st}}_{m|n}\sseq \g g$, where
$\g b^{\mathrm{st}}_{m|n}$ and 
$\g b^{\mathrm{op}}_{m|n}$
are defined in Appendix \ref{subsec-gl}. 
\\

\noindent\textbf{Case II.}
The matrix realization of  $\gflat\cong\g{osp}(4n|2m)$ 
is as in Appendix \ref{subsec-osp}. 
We set  
\[
h:=\mathrm{diag}\left(-I_{2n\times 2n},I_{2n\times 2n},-I_{m\times m},I_{m\times m}\right)
.
\] 
The matrix realization of 
$\g g\cong\gl(m|2n)$ is as in Appendix \ref{subsec-gl}, and 
the embedding $\g g\into\gflat $ is given by
\[
\begin{bmatrix}
A& B\\
C& D
\end{bmatrix}
\mapsto
\begin{bmatrix}
D & 0_{2n\times 2n} & C & 0_{2n\times m}\\
0_{2n\times 2n} & -D^T & 0_{2n\times m} & B^T\\
B & 0_{m\times 2n} & A & 0_{m\times m}\\
0_{m\times 2n} & -C^T & 0_{m\times m} &   -A^T
\end{bmatrix}.
\]
We set $\lag x,y\rag_\flat:=\mathrm{str}(xy)$.
The realization of 
$\g a_\eev$ as a subalgebra of $\gflat$ is
\[
\g a_\eev=\left\{
\mathrm{diag}
\left(\mathbf d,-\mathbf d,\mathbf a,-\mathbf a\right)\,:\,\mathbf d\in\C^m,\ \mathbf a:=(a_1,\ldots,a_{2n})\in\C^{2n},
a_{2i-1}=a_{2i}\text{ for }1\leq i\leq n
\right\},
\]
and thus $\mathit{\Sigma}$ is a root system of type
$\mathsf{A}(m-1,n-1)$ with 
$\underline\eps_i:=\eps_i\big|_{\g a_\eev}$, $1\leq i\leq m$, and $\underline{\delta}_j:=\delta_{2j-1}\big|_{\g a_\eev}$, $1\leq j\leq n$, where 
$\{\eps_i\}_{i=1}^m\cup\{\delta_j\}_{j=1}^{2n}$ are the standard characters of the Cartan subalgebra $\g h_{m|2n}\sseq \gl(m|2n)$, defined in \eqref{eq:CSAhm|n}. 
By a direct calculation we obtain $\lag \underline{\eps}_i,\underline{\eps}_j\rag_J^{}=-\frac{1}{2}\delta_{i,j}$ and 
$\lag \underline{\delta}_i,\underline{\delta}_j\rag_J^{}=\frac{1}{4}\delta_{i,j}$, so that $\theta_J=\frac{1}{2}$.

The embedding of $\g k\cong \g{osp}(m|2n)$ in $\g g$ is as in
Appendix \ref{composp}.
We set $\g b:=\g b_{m|2n}^{\mathrm{op}}$, where 
$\g b_{m|2n}^{\mathrm{op}}$ is defined in Appendix 
\ref{subsec-gl}. \\

\noindent\textbf{Case III.}  
Set $k:=\lfloor \frac{m+1}{2}\rfloor$.
The realization of $\gflat\cong\g{osp}(m+3|2n)$ is as in Appendix \ref{subsec-osp}. We set
\[
h:=
\begin{cases}
\mathrm{diag}(-2,0_{k\times k},2,0_{k\times k},0_{2n\times 2n})&\text{ if }\ m+1=2k,\\
\mathrm{diag}(0,-2,0_{k\times k},2,0_{k\times k},0_{2n\times 2n})&\text{ if }\ m+1=2k+1.
\end{cases}
\] 
The realization of $\g g\cong\g{gosp}(m+1|2n)$ is as in Appendix \ref{subsec-osp}.
Let $\{\eps_i\}_{i=1}^k\cup\{\delta_j\}_{j=1}^n\cup\{\zeta\}$
be the standard characters of the Cartan subalgebra 
$\tilde{\g h}_{m+1,n}$ of $\g g$. Then
$\g a_\eev=\bigcap_{i=2}^k\ker(\eps_i)\cap\bigcap_{j=1}^n\ker(\delta_j)$, and thus
$\mathit{\Sigma}=
\left\{\eps_1\big|_{\g a_\eev}\right\}$. We consider 
$\mathit{\Sigma}$ as a root system of type $\mathsf{A}(1,-1)$, where $\udl\eps_1-\udl\eps_2:=\eps_1\big|_{\g a_\eev}$ (the choice of $\udl\eps_1$ and $\udl\eps_2$ does not matter). Since there are no $\udl\delta_j$'s,
the value of $\theta_J$ is obtained only from the superdimension of $\udl\eps_1-\udl\eps_2$.

Similar to Appendix \ref{subsec-osp},
let $\{\sfe_i\}_{i=1}^{m+1}\cup\{\sfe'_j\}_{j=1}^{2n}$ 
denote the natural homogeneous basis of the standard $\g g$-module $\C^{m+1|n}$. 
Then $\g k\cong\g{osp}(m|2n)$ is the subalgebra of $\g g\cong\g{gosp}(m+1|2n)$ 
given by
\[
\g k:=\begin{cases}
\mathrm{Stab}_{\g g}(\sfe_1-\sfe_{k+1})
&\text{ if }m+1=2k,\\
\mathrm{Stab}_{\g g}(\sfe_2-\sfe_{k+2})&
\text{ if }m+1=2k+1.
\end{cases}
\]
The Borel subalgebra $\g b:=\g b_{m+1|2n}$
is defined in \eqref{bm|2ngosp}. \\

\noindent\textbf{Case IV.}
The Lie superalgebra $\gflat$ is isomorphic to 
Scheunert's Lie superalgebra 
$\Gamma(-t,-1,1+t)$
 (see \cite[Example I.1.5]{Scheunert}).
The realization of $\g g\cong\gl(1|2)$ is as in Appendix \ref{subsec-gl}, and 
% which is simple for $t\neq 0,-1$
% Note that$\mathit{D}_1\cong(1,2)_+$ and thus  $D(2|1,1)\cong\g{osp}(4|2)$. 
the 
embedding of $\g k\cong\g{osp}(1|2)$ in $\g g\cong\g{gl}(1|2)$ is 
as in 
Appendix \ref{composp}.
The Borel subalgebra $\g b:=\g b_{1|2}^{\mathrm{op}}$ is defined in Appendix \ref{subsec-gl}.

To identify $\mathit{\Sigma}$ and compute the value of $\theta_J$, we need the explicit realization of the root system of $\gflat$. 
To distinguish the root systems of $\gflat $ and $\g g$, we 
denote the root system of $\gflat $ by
$\Delta^\flat:=\Delta^\flat_\eev\sqcup\Delta^\flat_\ood$
where
\[
\Delta^\flat_\eev:=
\left\{
\pm2\tilde{\eps}_1,\pm2\tilde{\eps}_2,\pm2\tilde{\eps}_3
\right\}
\text{ and }
\Delta^\flat_\ood:=
\left\{
\pm\tilde{\eps}_1\pm\tilde{\eps}_2\pm\tilde{\eps}_3
\right\}.
\]
Let $(\cdot,\cdot)_\flat'$ denote the bilinear form induced on 
$\spn_\C\{\tilde{\eps}_i:1\leq i\leq 3\}$ by the 
invariant form $(\cdot,\cdot)_\flat$.
As usual, we choose $(\cdot,\cdot)_\flat$ such that  the $\tilde{\eps}_i$'s are orthogonal 
with respect to 
$(\cdot,\cdot)'_\flat$ and 
we have $(\tilde{\eps}_1,\tilde{\eps}_1)'_\flat=-t$, 
$(\tilde{\eps}_2,\tilde{\eps}_2)'_\flat=-1$,
and 
$(\tilde{\eps}_3,\tilde{\eps}_3)'_\flat=1+t$.
Let $\{h_i\}_{i=1}^3$ be a basis for the Cartan subalgebra of $\gflat$ that is dual to the $\tilde{\eps}_i$'s, i.e., $\tilde{\eps}_i(h_j)=\delta_{i,j}$. Then
$h:=h_1+h_2$ and therefore the fundamental roots of $\Delta$ are $\eps_1-\delta_1:=
\tilde{\eps}_1-\tilde{\eps}_2- \tilde{\eps}_3\big|_{\g h}$ and 
$\delta_1-\delta_2:=2\tilde{\eps}_3\big|_{\g h}$, where
$\g h$ denotes the diagonal Cartan subalgebra of $\g g$.
It follows that $\g a_\eev=\spn_\C\{h_1,h_2\}$, and 
therefore $\mathit{\Sigma}$ has only one odd root, hence it is of type $\mathsf{A}(0,0)$.
If we choose $\underline{\eps}_1:=\tilde{\eps}_1\big|_{\g a_\eev}$ then it follows that we should have $\underline{\delta}_1=
\tilde{\eps}_2\big|_{\g a_\eev}$, and by a straightforward calculation we obtain 
$\lag \underline{\eps}_1,\underline{\eps}_1\rag_J=-t$ and 
$\lag \underline{\delta}_1,\underline{\delta}_1\rag_J=-1$, so that $\theta_J=-\frac{1}{t}$. \\

\noindent\textbf{Case V.}
The embedding of 
$\g k
:=\g k^\mathrm{ex} \cong\g{osp}(1|2)\oplus\g{osp}(1|2)
$ 
in $\g g\cong\g{gosp}(2|4)$ is defined in Appendix \ref{appBexk}. 
The Borel subalgebra  $\g b:=\g b_{2|4}^{\mathrm{ex}}$ is defined in \eqref{b2|4ex}.

As in Case IV, we let $(\cdot,\cdot)_\flat'$ be the bilinear form induced 
on the dual of the Cartan subalgebra of $\gflat$
by 
the invariant form $(\cdot,\cdot)_\flat$ of $\gflat$.  The  root system
 $\Delta^\flat:=\Delta_\eev^\flat\sqcup
 \Delta_\ood^\flat$ of $\gflat$ is
 \[
\textstyle \Delta_\eev^\flat:=
 \left\{
 \pm\tilde{\eps}_i\pm\tilde{\eps}_j
 \right\}_{1\leq i<j\leq 3}\cup
\left\{
\pm\tilde{\eps}_i 
 \right\}_{i=1}^3
\cup\{\tilde\delta\} 
\text{ and }
\Delta_\ood^\flat:=
\left\{
\frac{1}{2}\big(\pm\tilde{\eps}_1
\pm\tilde{\eps}_2
\pm\tilde{\eps}_3
\pm\tilde{\delta}
\big)
\right\},
\]
such that $(\tilde{\eps}_i,\tilde{\eps}_j)'_\flat=\delta_{i,j}$, $(\tilde\delta,\tilde\delta)'_\flat=-3$, and $(\tilde{\eps}_i,\tilde\delta)'_\flat=0$.
Let $\left\{h_{\tilde{\eps}_i}\right\}_{i=1}^3\cup\left\{h_{\tilde\delta}\right\}$
be a basis 
 dual to $\left\{\tilde{\eps}_i\right\}_{i=1}^3\cup
\{\tilde{\delta}\}$
for the Cartan subalgebra of $\gflat$. 
We set $h:=h_{\tilde{\eps}_1}+h_{\tilde\delta}$.
Then the fundamental roots of $\Delta$ are
\[
\textstyle
\eps_1-\delta_1:=\frac{1}{2}(\tilde\delta-\tilde{\eps}_1
-\tilde{\eps}_2
-\tilde{\eps}_3)\big|_{\g h},\ 
\delta_1-\delta_2:=\tilde{\eps}_3\big|_{\g h},
\text{ and }
2\delta_2:=(\tilde{\eps}_2-\tilde{\eps}_3)\big|_{\g h}.
\]
From the description of $\g k$ it follows that 
$\delta_1-\delta_2$ is a root of $\g k$, hence  
$\delta_1-\delta_2\big|_{\g a}=0$. Consequently, 
\[
\g a_\eev=\spn_\C\left\{
h_{\tilde\delta},h_{\tilde{\eps}_1},h_{\tilde{\eps}_2}
\right\}.
\]
One can now verify  that $\mathit{\Sigma}$ is a root system of type $\mathsf{A}(1,0)$, with fundamental roots 
\[
\textstyle\underline{\eps}_1-\underline{\eps}_2:=\tilde{\eps}_2\big|_{\g a_\eev}\text{ and }\underline{\eps}_2-\underline{\delta}_1:=
\frac{1}{2}(\tilde{\delta}-\tilde{\eps}_1-
\tilde{\eps}_2-\tilde{\eps}_3)\big|_{\g a_\eev}.
\]
We can determine the value of $\theta_J$ without making a choice for the $\underline{\eps}_i$ and the $\underline{\delta}_j$, as follows. 
First note that 
$\underline{\eps}_1-\underline{\eps}_2=\tilde\eps_2\big|_{\g a}$, so that 
\begin{equation}
\label{eps1eps2eq}
\lag\underline{\eps}_1-\underline{\eps}_2,\underline{\eps}_1-\underline{\eps}_2\rag_J^{}=
(\tilde\eps_2,\tilde\eps_2)_\flat^{}=1.
\end{equation}
 Since $\mathit{\Sigma}$ is assumed to be a Sergeev-Veselov deformed root system, in particular we should have
$\lag\underline{\eps}_1,\underline{\eps}_1\rag_J^{}=
\lag\underline{\eps}_2,\underline{\eps}_2\rag_J^{}$ and 
$\lag\underline{\eps}_1,\underline{\eps}_2\rag_J^{}=0$.
Thus from \eqref{eps1eps2eq} it follows that 
$\lag \underline\eps_1,\underline\eps_1\rag_J^{}=\frac{1}{2}$.
Similarly, 
$\underline{\eps}_1-\underline{\delta}_1=\frac{1}{2}(\tilde{\delta}-\tilde{\eps}_1+\tilde{\eps}_2)\big|_{\g a_\eev}$, so that
\[\textstyle
\lag\underline{\eps}_1,\underline{\eps}_1\rag_J+
\lag \underline{\delta}_1,\underline{\delta}_1\rag_J^{}
=
\lag \underline{\eps}_1-\underline{\delta}_1,
\underline{\eps}_1-\underline{\delta}_1
\rag_J^{}
=\frac{1}{4}
(\tilde{\delta}-\tilde{\eps}_1+\tilde{\eps}_2,
\tilde{\delta}-\tilde{\eps}_1+\tilde{\eps}_2
)_\flat^{}=-\frac{1}{4}.
\]
Consequently, 
$\lag \underline{\delta}_1,\underline{\delta}_1\rag_J^{}
=
-\frac{1}{4}-\lag \underline{\eps}_1,\underline{\eps}_1\rag_J^{}
=-\frac{3}{4}$. From the values of  
$
\lag \underline{\eps}_1,\underline{\eps}_1\rag_J^{}
$
and
$\lag \underline{\delta}_1,\underline{\delta}_1\rag_J^{}
$ 
we obtain
$\theta_J=\frac{3}{2}$.\\

\noindent\textbf{Case VI.}
The realization of $\gflat\cong\g p(2n)$ is as in
Appendix \ref{subsec-p}. We set
\[
h:=\mathrm{diag}(-I_{n\times n},
I_{n\times n},
I_{n\times n}
-I_{n\times n}).
\]
The realization of $\g g\cong\gl(n|n)$ is as in Appendix \ref{subsec-gl}, and the embedding 
$\g g\into\gflat$ is given by the map
\[
\begin{bmatrix}
A& B\\
C& D
\end{bmatrix}
\mapsto
\begin{bmatrix}
A & 0_{n\times n}& 0_{n\times n} & B\\
0_{n\times n} & -D^T & B^T & 0_{n\times n}\\
0_{n\times n} & -C^T & -A^T & 0_{n\times n}\\
C & 0_{n\times n} & 0_{n\times n} & D
\end{bmatrix}.
\]
The realization of $\g a_\eev$ as a subalgebra 
of $\gflat$ is 
\[
\g a_\eev :=\left\{
\mathrm{diag}(\mathbf a,-\mathbf d,-\mathbf a,\mathbf d)\,:\,\mathbf a,\mathbf d\in\C^n\right\},
\]
and $\mathit{\Sigma}$ is a root system of type $\mathsf{Q}(n)$.
The embedding of  $\g k\cong\g{p}(n)$ in $\g g\cong\gl(n|n)$ is given in 
Appendix \ref{subsec-p}.
%consists of matrices in $(n,n)$-block form\[\begin{bmatrix}A & B\\ C& -A^T\end{bmatrix}\quad\text{ where }B=-B^T\text{ and }C=C^T.\]
The Borel subalgebra $\g b:={\g b}^\mathrm{mx}_{n|n}$ is defined  in Appendix \ref{subsec-gl}.
\\

\noindent\textbf{Case VII.} 
The matrix realization of $\gflat\cong\g q(2n)$ is as
in Appendix \ref{subsec-q}. The embedding of $\g g\cong\g q(n)\oplus\g q(n)$ in $\gflat$ is the restriction of the one given in Case I. The subalgebra $\g a_\eev$ is the intersection with $\g g$   
of the one given in Case I.
The
embedding of $\g k\cong\g{q}(n)$ in
$\g g\cong\g{q}(n)\oplus\g{q}(n)$ is also the restriction of the diagonal map $x\mapsto x\oplus x$.
We set $\g b:=\g b^{\mathrm{op}}_n\oplus\g b^{\mathrm{st}}_n$, where
$\g b^{\mathrm{st}}_n$ and $\g b^{\mathrm{op}}_n$ are defined
in Appendix \ref{subsec-q}. \\
%The highest weight $\lambda_{n}^{\mathrm{st}}$ that appears in Table \ref{tbl-3}is also defined in Appendix \ref{subsec-q}.\\

We summarize the descriptions of $\gflat$, $\g g$, $\g k$, and $\g b$ in 
Table \ref{tbl-1}. In addition, in the last column of Table \ref{tbl-1} we give an explicit realization of $V$ for Cases I--III and VI--VII, and the $\g b$-highest weight of $V$ for Cases IV and V.
The symbol $\Pi$ in Cases VI and VII of Table \ref{tbl-1} is the parity reversal functor,  
so that 
\[((\C^{n|n})^*\otimes \C^{n|n})^{\Pi\otimes \Pi}:=
\left\{
v\in((\C^{n|n})^*\otimes \C^{n|n})\,:\,(\Pi\otimes\Pi)(v)=v\right\}.
\]

\begin{table}[h]
\footnotesize

\begin{tabular}{lcccccc}
&
$\gflat$ & 
$\g g$ & 
$\g k$ & 
$\g b$ &
$V$\\
\hline 

$\mathrm{I}$&
$\g{gl}(2m|2n)$&
$\g{gl}(m|n)\oplus\g{gl}(m|n)$&
$\g{gl}(m|n)$&
$\g b_{m|n}^{\mathrm{op}}\oplus
\g b_{m|n}^{\mathrm{st}}
$&
$\C^{m|n}\otimes(\C^{m|n})^*$
\\

$\mathrm{II}$&
$\g{osp}(4n|2m)$ & 
$\g{gl}(m|2n)$ &
$\g{osp}(m|2n)$ &
$\g b_{m|2n}^\mathrm{op}$ &
$\sS^2(\C^{m|2n})$\\

$\mathrm{III}$&
$\g{osp}(m+3|2n)$& 
$\g{gosp}(m+1|2n)$& 
$\g{osp}(m|2n)$ & 
$\g b_{m+1|2n}$ &
$(\C^{m+1|2n})^*$
\\

$\mathrm{IV}$&
$D(2|1,t)$&
$\gl(1|2)$&
$\g{osp}(1|2)$& 
$\g b_{1|2}^{\mathrm{op}}$ &
$\left(-\frac{3+t}{1+t}\right)\eps_1+
\left(\frac{2+t}{1+t}\right)
(\delta_1+\delta_2)
$
\\

$\mathrm{V}$&
$F(3|1)$&
$\g{gosp}(2|4)$&
$\g{osp}(1|2)\oplus
\g{osp}(1|2)$&
$\g b_{2|4}^\mathrm{ex}$&
$-3\eps_1-\zeta$

\\

$\mathrm{VI}$&
$\g p(2n)$ &
$\g{gl}(n|n)$&
$\g p(n)$& 
$\g b_{n|n}^\mathrm{mx}$ &
$\Pi(\Lambda^2(\C^{n|n}))$\\

$\mathrm{VII}$&
$\g q(2n)$ &
$\g{q}(n)\oplus\g{q}(n)$&
$\g q(n)$& 
$\g b^\mathrm{op}_n\oplus 
\g b^\mathrm{st}_n$
&
$(\C^{n|n}\otimes (\C^{n|n})^*)^{\Pi\otimes \Pi}$\\

\hline
\end{tabular}

\smallskip

\smallskip

\smallskip

\smallskip

\caption{The quintuples 
$(\gflat,\g g,\g k,\g b,V)$.}
\label{tbl-1}

\end{table}

\section{Proof of Theorem \ref{thm:compl-red}}
\label{sec-mulfree}
In this section we prove Theorem \ref{thm:compl-red}. 
Recall that by Remark \ref{rmk-someprvd}, we will only need to consider Cases IV--VI. 
We address each case separately.

\subsection{Case IV}

\label{CRCaseIV}
Recall that in this case $\g g\cong\gl(1|2)$. Let  $\g b_{1|2}^\mathrm{st}$ be the Borel subgroup of $\g g$ defined as in Appendix \ref{subsec-gl}.
The next proposition implies Theorem \ref{thm:compl-red} in Case IV. 
\begin{prp}
\label{prp-D21tKss}
Assume that $J\cong \mathit{D}_t$ for $t\neq 0,-1$.  Then $\sP(V)$ is multiplicity-free  if and only if 
$-\frac{1}{t}\not\in\Q^{\leq 0}$. If the latter condition holds, 
 then for every $d\geq 1$ we have
  \[
\sP^d(V)
\cong\bigoplus_{k=1}^d 
V_{\eta_k},
 \]
where $V_{\eta_k}$ is the irreducible $\g g$-module with $\g b$-highest weight 
\[\eta_k:=
\left(d\left(\frac{3+t}{1+t}\right)-2k\right)\eps_1+\left(
-d\left(\frac{2+t}{1+t}\right)+k
\right)(\delta_1+\delta_2)
.
\]
\end{prp}

\begin{proof}
By the $\g g$-module isomorphism
$\sP^d(V)\cong\sS^d(V^*)$ it is  enough to prove the analogous statement for $\sS^d(V^*)$.
The $\g b^{\mathrm{st}}_{1|2}$-highest weight
of  
$V^*$ is \[
\eta:=\left(\frac{3+t}{1+t}\right)\eps_1-
\left(\frac{2+t}{1+t}\right)
(\delta_1+\delta_2)
.
\]
%Set
%$\g l:=\g g\cong\gl(1|2)$ and $\g b_{\g l}:=\g b_{1|2}^{\mathrm{st}}$ (this uniquely determines $\g p_\g l$).
For every 
$\mu:=x_1\eps_1+y_1\delta_1+y_2\delta_2\in\g h^*_\eev$ such that $y_1-y_2\in\Z^{\geq 0}$, we denote the irreducible finite dimensional $\g g_\eev^{}$-module with $(\g b_{1|2}^{\mathrm{st}}\cap \g g_\eev^{})$-highest weight $\mu$ by $M_\mu$. Let
$K(\mu):=\mathrm{Ind}_{\g b_{1|2}^\mathrm{st}}^\g g M_\mu$ be the corresponding Kac module.
%Observe that $\g g_\eev\cong \gl(1,\C)\oplus\gl(2,\C)$. 
 As a $\g g_\eev^{}$-module,  $K(\mu)_\ood$ is isomorphic to $M_\mu\otimes M_{-\eps_1+\delta_1}$.
Therefore 
\begin{equation}
\label{Kmux1--}
K(\mu)_\ood\cong
\begin{cases}
M_{\mu-\eps_1+\delta_1}& \text{ if }y_1=y_2,\\
M_{\mu-\eps_1+\delta_1}\oplus 
M_{\mu-\eps_1+\delta_2}
& \text{ if }y_1>y_2.
\end{cases}
\end{equation}
Furthermore, $\mu$ is a typical 
$\g b_{1|2}^{\mathrm{st}}$-highest weight if and only if $x_1+y_1\neq 0$ and $x_1+y_2\neq 1$.

Since $\eta$ is typical, we have $V^*\cong K(\eta)$.
Let $\mathcal F$ denote the category of finite dimensional $\g h$-weight modules of $\g g$. Typicality of $\eta$ implies that $V^*$ is projective in 
$\mathcal F$. Consequently, the tensor product of $V^*$ and any object of $\mathcal F$ is also projective 
(the proof of the latter statement is similar to \cite[Prop. 3.8(b)]{HumphreysBGG}). It follows that
$\sS^d(V^*)$, which is a submodule of $(V^*)^{\otimes d}$, is also projective, and therefore it
has a filtration by Kac modules
(see \cite[Prop. 2.5]{Zou}). 

Set $\gamma:=\eta-2\eps_1+\delta_1+\delta_2$
and $\gamma_k:=
d\eta-(2k+1)\eps_1+(k+1)\delta_1+k\delta_2$ 
for $0\leq k\leq d-1$.
 Then  we have an isomorphism of $\g g_\eev^{}$-modules
\begin{align}
\label{SdKmut}
\sS^d(V^*)_\ood
\notag
&
\cong 
\sS^{d-1}(V^*_\eev)\otimes V^*_\ood\\
&
\cong
\sS^{d-1}\left(M_\gamma\oplus M_{\eta}\right)\otimes M_{\eta-\eps_1+\delta_1}
\cong\bigoplus_{k=0}^{d-1}M_{k\gamma+(d-k-1)\eta}\otimes M_{\eta-\eps_1+\delta_1}\cong
\bigoplus_{k=0}^{d-1}
M_{\gamma_k}.
\end{align}
By comparing \eqref{SdKmut} with \eqref{Kmux1--}, it follows that  the Kac-module filtration of $\sS^d(V^*)$  
consists of exactly one copy of  each of the modules
$K(\gamma_k+\eps_1-\delta_1)$.
If $\gamma_k+\eps_1-\delta_1$ is atypical for some $k$, then $\sS^d(V^*)$ cannot be completely reducible because the subquotient $K(\gamma_k+\eps_1-\delta_1)$ is reducible but indecomposable. Thus, a necessary condition for complete reducibility of $\sS^d(V^*)$ is that 
$\gamma_k+\eps_1-\delta_1$ is typical for every $0\leq k\leq d-1$. But the latter necessary condition is also sufficient because typical modules always split off as direct summands.

Next we determine when all of the $\gamma_k+\eps_1-\delta_1$ are typical. Note that 
\[
\gamma_k+\eps_1-\delta_1=
\left(d\left(\frac{3+t}{1+t}\right)-2k\right)\eps_1+\left(
-d\left(\frac{2+t}{1+t}\right)+k
\right)(\delta_1+\delta_2)
\quad\text{ for }\quad 0\leq k\leq d-1.
\]
Therefore 
$\gamma_k+\eps_1-\delta_1$ 
is typical for all  $0\leq k\leq d-1$  if and only if $\frac{d}{1+t}\not\in\{0,\ldots,d\}$. It follows that $\sS(V^*)$ is completely reducible if and only if 
$\frac{1}{1+t}\not\in
\left\{x\in\Q\,:\,0\leq x\leq 1\right\}$.
Since $t\neq 0$, the latter condition 
 can be expressed as $-\frac{1}{t}\notin \Q^{\leq 0}$. 
 
 Finally, using 
the fact that 
 $\g b$ is obtained from $\g b_{1|2}^{\mathrm{st}}$ by the composition 
$s_{\delta_1-\delta_2}\circ s_{\eps_1-\delta_2} \circ s_{\eps_1-\delta_1} 
$  of even and odd reflections, it is straightforward to 
verify that the $\g b$-highest weight of $K(\gamma_k+\eps_1-\delta_1)$ is $\eta_{k+1}$ (see \cite[Lem 1.40]{CWBook}).
\end{proof}

\subsection{Case V}

For $d\geq 0$ we have a $\g g$-module isomorphism $\sP^d(V)\cong \sS^d(V^*)$. The weights of $V^*_\eev$ are 
$
\{\eps_1+\zeta,\eps_1\pm\delta_1\pm\delta_2+\zeta,3\eps_1+\zeta\}
$
and the weights of $V^*_\ood$ are $\{2\eps_1\pm\delta_1+\zeta,
2\eps_1\pm\delta_2+\zeta\}$.

Let $\g b:=\g b_{2|4}^{\mathrm{ex}}$ and
$\g b_{2|4}$ be the Borel subalgebras that are chosen in Appendix  \ref{subsec-osp}. Then $\g b$ can be obtained from $\g b_{2|4}$ by applying the sequence of odd reflections 
\begin{equation}
\label{eq:bbexoddref}
r_{\eps_1+\delta_1}\circ
r_{\eps_1+\delta_2}\circ
r_{\eps_1-\delta_2}\circ
r_{\eps_1-\delta_1}.
\end{equation}
Let $u$ be a $\g b$-highest weight vector of $V^*$. Then $u$ has weight $\eps_1+\delta_1+\delta_2+\zeta$. 
Also, let $w$ be a $\g b_{2|4}$-highest weight vector of $V^*$. Then 
 the weight of $w$ is $3\eps_1$, hence 
 $w\in V_\eev^*$ and therefore $w^k$ for $k\geq 2$ is a $\g b_{2|4}$-highest weight vector in $\sS^k(V^*)$ of weight 
$3k\eps_1$. 
For $s\geq 2$, we  set \[
w_s:=e_{-\eps_1-\delta_1}
(e_{-\eps_1-\delta_2}
(e_{-\eps_1+\delta_2}
(e_{-\eps_1+\delta_1}(w^s)
))),\]
where the $e_{-\eps_1\pm \delta_i}$ denote root vectors of $\g g$.  
\begin{lem}
For $s\geq 2$, the vector $w_s$ is a typical $\g b$-highest weight, whose  weight is $(3s-4)\eps_1+s\zeta$. In addition, $w_2\in\sS^2(V^*_\eev)$. 
\end{lem}

\begin{proof}
It is straightforward to verify that $3s\eps_1+s\zeta$  is a typical $\g b_{2|4}$-highest weight. From the relation between $\g b$ and $\g b_{2|4}$ via odd reflections
given in \eqref{eq:bbexoddref}, it follows that 
 the $\g b$-highest weight vector of the irreducible summand of $\sS^s(V^*)$ generated by $w^s$ is $w_s$ (see \cite[Lem. 1.40]{CWBook}). 
Since the $(2\eps_1+2\zeta)$-weight space of $\sS^2(V^*)$ is indeed a subspace of $\sS^2(V_\eev^*)$, we obtain $w_2\in\sS^2(V^*_\eev)$.
\end{proof}

\begin{dfn}
\label{def-Edwts}
Let $\mathcal W_d$ denote the set of $\g b$-highest weight vectors given in (i) and (ii) below.
\begin{itemize}
\item[(i)] Vectors of the form
$u^q(w_2)^rw_s$ for integers $q,r,s$ that satisfy $q,r\geq 0$, $s\geq 2$, and $d+2r+s=d$. Note that
$u^q(w_2)^rw_s\neq 0$ since $u,w_2\in \sS(V^*_\eev)$.
The weight of 
$u^q(w_2)^rw_s$ is equal to 
\begin{equation}
\label{d+2s-4eps1+bd}
(d+2s-4)\eps_1+(d-2r-s)\left(\delta_1+\delta_2\right)+d\zeta.
\end{equation}
\item[(ii)] The vector $u^d$, whose weight is $d\eps_1+d\delta_1+d\delta_2+d\zeta$.

\end{itemize}
We denote the set of weights of the vectors in $\mathcal W_d$ by $\mathcal E_d$. 
\end{dfn}
Let $\g m$ be the image of the unique embedding 
$\iota:\g{sp}(4)\to \g g_\eev$.
For every $\mu:=y_1\delta_1+y_2\delta_2$ such that $y_1\geq y_2$, let $M(\mu)$ denote the irreducible 
$\g m$-module with $\g b\cap \g m$-highest weight $\mu$. 

\begin{lem}
\label{lem:ddd-1d-1}
Let  $W$ be the irreducible $\g g$-module with $\g b$-highest weight $d\eps_1+d\delta_1+d\delta_2+d\zeta$, for $d\geq 1$. Then both $M\big(d\delta_1 +d\delta_2)$ and
$M\big((d-1)\delta_1+(d-1)\delta_2
\big)$
occur as $\g m$-submodules of  $W_\eev$. 
\end{lem}
\begin{proof}
The $\g m$-submodule of $W$ that is generated by the $\g b$-highest weight of $W$ is isomorphic to 
$M\big(d\delta_1 +d\delta_2)$.
From
\cite[Lem. 1.40]{CWBook} and
the relation between $\g b$ and  $\g b_{2|4}$ via odd reflections given in \eqref{eq:bbexoddref}
it follows that the $\g b_{2|4}$-highest weight  of $W$ is 
$(d+2)\eps_1+(d-1)\delta_1+(d-1)\delta_2+d\zeta$. Since
$\g b_{2|4}\cap \g m=\g b\cap\g m$, 
the $\g m$-module generated by the 
$\g b_{2|4}$-highest weight  of $W$ 
 is isomorphic to $M\big((d-1)\delta_1
+(d-1)\delta_2\big)$.
\end{proof}

\begin{lem}
\label{sp4lem-SC}
For every $k\geq 0$ there is an isomorphism of $\g{sp}(4)$-modules
\[\sS^k\big(M(\delta_1+\delta_2)\big)\cong
\bigoplus_{i=0}^{\lfloor \frac{k}{2}\rfloor}M\big((k-2i)\delta_1+(k-2i)\delta_2\big) 
.
\]
\end{lem}
\begin{proof}
Note that $\g{sp}(4)\cong \g{so}(5)$ and
$M(\delta_1+\delta_2)$ is the standard $5$-dimensional representation of  $\g{so}(5)$. The statement now follows from the classical theory of spherical 
harmonics (for example see 
\cite[Thm 5.6.11]{GoodmanWallach}).
\end{proof}

The next proposition proves  Theorem \ref{thm:compl-red} in Case V. 
\begin{prp}
Assume that $J\cong\mathit{F}$. Then 
$\sP^d(V)$ is a multiplicity-free direct sum of irreducible $\g g$-modules with $\g b$-highest weights in $\mathcal E_d$, where
$\mathcal E_d$ is as in Definition \ref{def-Edwts}.  
\end{prp}

\begin{proof}

For every $\gamma\in \mathcal E_d$,
we denote the irreducible $\g g$-module with $\g b$-highest weight $\gamma$ by $W_\gamma$.
If $\gamma\neq d\eps_1+d\delta_1+d\delta_2+d\zeta$, then
by setting 
$a:=s-2$ and $b:=d-2r-s$
in \eqref{d+2s-4eps1+bd}
we can express $\gamma$  as 
\begin{equation}
\label{eq:hwabessa}
\gamma:=(d+2a)\eps_1+b\delta_1+b\delta_2+d\zeta,
\end{equation}
where $0\leq a\leq d-2$, $0\leq b\leq d-a-2$ and 
 $b\equiv d-a\ (\text{mod}\ 2)$.
It is straightforward to verify that every $\gamma$ 
of the form 
\eqref{eq:hwabessa}
is typical. 
 Since typical submodules split off as direct summands, in order to prove the assertion of the proposition it suffices to show that the only irreducible subquotients of $\sP^d(V)$ are those whose $\g b$-highest weight vectors are in $\mathcal W_d$.

For  $\gamma$ of the 
form \eqref{eq:hwabessa}, typicality of $\gamma$ implies that 
 $W_\gamma$ is a Kac module, and therefore $\dim (W_\gamma)_\eev=8\dim M(b\delta_1+b\delta_2)$. Set $N_k:=
\dim M(k\delta_1+k\delta_2)$. By the Weyl character formula  for $\g{sp}(4)$, 
\begin{equation}
\label{Welysp4}
N_k=\frac{(2k+3)(k+2)(k+1)}{6}
\text{ for }k\geq 0.
\end{equation}
From \eqref{Welysp4}, Lemma \ref{lem:ddd-1d-1}, and 
Lemma \ref{sp4lem-SC}
 it follows that
\begin{align*}
\sum_{\gamma\in\mathcal E_d}\dim (W_\gamma)_\eev&=
N_d+N_{d-1}+
8\sum_{a=0}^{d-2}
\sum_{\tiny\begin{array}{l}
0\leq b\leq d-a-2\\
b\equiv d-a\ (\text{mod}\ 2)
\end{array}}\dim M(b\delta_1+b\delta_2)\\
&
=\frac{(d+1)(2d^2+4d+3)}{3}+8\sum_{a=0}^{d-2}
\dim \sS^{d-a-2}(\C^5)\\
&=\frac{(d+1)(2d^2+4d+3)}{3}+8\dim\sS^{d-2}(\C^6)\\
&=\frac{(d+1)(2d^2+4d+3)}{3}+\frac{(d+3)(d+2)(d+1)d(d-1)}{15}=\dim \sS^d(V^*)_\eev.
\end{align*}
The relation
$\sum_{\gamma\in\mathcal E_d}\dim (W_\gamma)_\eev
=\dim \sS^d(V^*)_\eev$ implies that if $\sS^d(V^*)$ has an irreducible subquotient $W$ such that $W\not\cong W_\gamma$ for all $\gamma\in\mathcal E_d$, then 
$W_\eev=\{0\}$. Since $[\g g,\g g]$ is generated by its odd elements, $W$ must be a trivial $[\g g,\g g]$-module. However, the $\g h\cap[\g g,\g g]$-weights of $V^*$ are 
$-\eps_1\pm\delta_1\pm \delta_2$, $-\eps_1$, $-3\eps_1$,  $-2\eps_1\pm\delta_1$, and $-2\eps_1\pm\delta_2$, and it is straightforward to verify that all of the $\g h\cap[\g g,\g g]$-weights of $\sS^d(V^*)$ are nonzero, which is a contradiction. 
Consequently, $\sS^d(V^*)\cong\bigoplus_{\gamma \in\mathcal E_d}W_\gamma$.
\end{proof}

\subsection{Case VI}

For $d\geq 0$ we have an 
isomorphism of $\g g$-modules
\begin{equation}
\label{eqCVI-sPdLam}
\sP^d(V)\cong \sS^d(\Pi(\Lambda^2(\C^{n|n}))^*)
\cong \left(\Pi^d\Lambda^d(\Lambda^2(\C^{n|n}))\right)^*.
\end{equation}
Therefore in Case VI, Theorem \ref{thm:compl-red} follows
from the following proposition.
\begin{prp}
For every $d\geq 0$ we have
\[
\left(\Pi^d\Lambda^d(\Lambda^2(\C^{n|n}))\right)^*
\cong \bigoplus_{\mu\in\EuScript{DP}_d(n)}
V_\mu,
\]
where
$V_\mu$ denotes the $\gl(n|n)$-module with
$\g b_{n|n}^{\mathrm{mx}}$-highest weight
$-\sum_{i=1}^n\mu_i(\eps_i+\delta_i)$. 
\end{prp}
\begin{proof}
By 
Schur--Weyl--Sergeev duality, 
\begin{equation}
\label{eq:Sergduall}
(\C^{n|n})^{\otimes 2d}
\cong
\bigoplus_{\mu\in\EuScript H_{2d}(n,n)}
E_{\mu_{n|n}^{\mathrm{st}}}\otimes F_\mu,
\end{equation}
where $E_{\mu_{n|n}^{\mathrm{st}}}$ is the $\gl(n|n)$-module with $\g b_{n|n}^{\mathrm{st}}$-highest weight $\mu_{n|n}^\mathrm{st}$ defined as in \eqref{eq:lammn-stt}
and $F_\mu$ is  the 
$S_{2d}$-module associated  to the partition  $\mu$ in the standard way.

For every $\lambda\in \EuScript{DP}(n)$ we define $\breve\lambda\in\EuScript{H}(n,n)$ 
to be the partition whose Young diagram is  constructed by nesting the $(1,1)$-hooks with $\lambda_i$ boxes in the first row and $\lambda_i+1$ boxes in the first column, where $1\leq i\leq \ell(\lambda)$.
For example, if $\lambda=(4,2,1,0,\ldots)$ 
then $\breve\lambda=(4,3,3,1,0,\ldots)$.
From
\eqref{eq:Sergduall} 
and a superized variation of the proof of \cite[Thm 4.4.4]{HoweSchur} we obtain 
\begin{equation}
\label{Lamkadccps}
\Lambda^d(
\sS^2(\C^{n|n}))\cong
\bigoplus_{\mu\in\EuScript{DP}(n)_d}
E_{(\breve\mu)_{n|n}^\mathrm{st}}.
\end{equation}
By comparing \eqref{eqCVI-sPdLam} and 
\eqref{Lamkadccps}, we obtain that
\[
\sP^d(V)\cong 
\bigoplus_{\mu\in\EuScript{DP}(n)_d}
W_\mu,
\]
where $W_\mu$ is the $\gl(n|n)$-module with $\g b^{\mathrm{op}}_{n|n}$-highest weight  $-(\breve\mu)_{n|n}^{\mathrm{st}}$.

Set $\lambda:=\breve\mu$, and recall that $\lag x\rag :=\max\{x,0\}$ for $x\in\R$. A straightforward calculation based on the method of 
\cite[Sec. 2.4.1]{CWBook}
shows that the 
$\g b_{n|n}^\mathrm{mx}$-highest weight of 
$W_\mu$ is
\[
-\sum_{i=1}^n\lag \lambda_i-i+1\rag\eps_i
-
\sum_{i=1}^n\lag \lambda_i'-i\rag\delta_i=-\sum_{i=1}^n\mu_i(\eps_i+\delta_i).  
\qedhere\]
\end{proof}

\section{Surjectivity of the Harish-Chandra map}
\label{Sec:surjectivity}

This section is devoted to the proof of Theorem \ref{thm:surjZg}. The proof is divided into two parts: the case $J\not\cong \mathit{F}$, given in Proposition 
\ref{prp-JresHCZiLJ},
and the case
$J\cong \mathit{F}$, 
given in  Proposition 
\ref{prp:typeFresHCsseqtauJ*}.

The associative superalgebra $\sPD(V)$  has a natural filtration 
$\left(\sPD^{(i)}(V)\right)_{i\geq 0}$
by order of operators.  
We denote the natural degree filtrations of $\sP(\g h^*)$ and $\sP(\aomega)$  by
$\left(\sP^{(i)}(\g h^*)\right)_{i\geq 0}$ and 
$\left(\sP^{(i)}(\aomega)\right)_{i\geq 0}$, respectively. 
 
In the following lemma,  $B_t(z):=\sum_{i=0}^t\frac{1}{i+1}
\sum_{j=0}^i(-1)^j {i \choose j}(z+j)^t$ denotes
the Bernoulli polynomial of degree $t$. 

\begin{lem}\label{lem:htBt-theL}
Let $k,k'\in\N$, and let $\theta$ be a complex number   such
that $\theta\not\in\Q^{\leq 0}$. 
For $t\in\N$, let $h_t(x,y,\theta)\in\Lambda_{k,k',\theta}^\natural$ be defined by
\begin{align*}
\textstyle 
h_t(x,y,\theta)&
\textstyle 
:=\sum_{i=1}^k 
B_t\left(
x_i+\frac{1}{2}
\right)
+(-\theta)^{t-1}\sum_{j=1}^{k'}
B_t\left(
y_j+\frac{1}{2}\right)
.
\end{align*}
Finally, let $\left(
\Lambda_{k,k',\theta}^{\natural,(i)}
\right)_{i\geq 0}
$ denote the degree filtration of 
$\Lambda_{k,k',\theta}^\natural$.
Then 
for every $d\geq 0$ we have
\begin{equation}
\label{eq:btBt}
\Lambda_{k,k',\theta}^{\natural,(d)}
=\spn_\C\left\{
h_1^{m_1}\cdots h_d^{m_d}\ :\
m_1,\ldots,m_d\in\Z^{\geq 0}\text{ and }\sum_{j=1}^d jm_j\leq d
\right\}.
\end{equation}
\end{lem}
\begin{proof}
Clearly 
$h_t\in \Lambda_{k,k',\theta}^{\natural,(t)}$ for  
$t\in\N$. Therefore it remains to prove 
that in
\eqref{eq:btBt},
the left hand side is a subset of  the right hand side.
For $N\in\N$, let $\Lambda_{N,\theta}$ denote
the $\C$-algebra of polynomials $f(z_1,\ldots,z_N)$ which are 
symmetric in $z_i+\theta(1-i)$, and 
let $\Lambda_\theta:=\varprojlim\Lambda_{N,\theta}$ be the algebra of $\theta$-shifted symmetric functions.
Let 
$\left(\Lambda_{\theta}^{(i)}\right)_{i\geq 0}$
be the  degree filtration of $\Lambda_\theta$.
For $t\in\N$, let $b_t(z,\theta)\in\Lambda^{(t)}_\theta$ be defined by 
\[
b_t(z,\theta):=
\sum_{i=1}^\infty
\left[
B_t\left(
z_i+\frac{1}{2}+\theta\left(
\frac{1}{2}-i
\right)
\right)
-B_k
\left(
\frac{1}{2}+\theta\left(\frac{1}{2}-i\right)
\right)
\right].
\]
In 
\cite[Sec. 6]{SVSuperJack},
Sergeev and Veselov construct an epimorphism  $\varphi^\natural:\Lambda_\theta\to\Lambda^{\natural}_{k,k',\theta}$ such that
\begin{equation}
\label{prpsofphinat}
\varphi^\natural\left(\Lambda_\theta^{(t)}\right)=\Lambda_{k,k',\theta}^{\natural,{(t)}}\text{ and } 
\deg\left(\varphi^\natural(b_t)-h_t\right)<t
\text{ for }t\in\N.
\end{equation}
Since $\mathrm{gr}(\Lambda_{\theta})$ is isomorphic to the algebra of symmetric functions, the formal series $b_t^\natural$ constitute algebraically independent generators of $\Lambda_\theta$. Consequently,  
\begin{equation}
\label{eq:lamthedd}
\Lambda_\theta^{(d)}\sseq\spn_\C
\left\{b_1^{m_1}
\cdots
b_d^{m_d}
\ :\ 
m_j\in\Z^{\geq 0}\text{ for $j\geq 1$  and }
\sum_{j=1}^djm_j\leq d
\right\}.
\end{equation} 
From \eqref{eq:lamthedd} and \eqref{prpsofphinat}
 it follows that 
$\Lambda_{k,k',\theta}^{\natural,{(d)}}$ is a subset of the right hand side of 
\eqref{eq:btBt}.
  \end{proof}

\begin{prp}
\label{prp-JresHCZiLJ}
Assume that  $J\not\cong\mathit{F}$. 
Then   
\begin{equation}
\label{resHCsseqtauJ}
\mathsf{res}\left(\HC\left(\bfZ^{(i)}(\g g)\right)\right)= \tau_J^*\left(\Lambda_J^{(i)}\right)\text{ for }i\geq 0.
\end{equation}
\end{prp}

\begin{proof}
The idea of the proof for all of the cases is similar and uses the explicit  Harish-Chandra homomorphism (see for example
\cite[Prop. 2.25]{CWBook}). However, the  explicit calculations in each case is  different. Therefore our analysis is case by case.\\

\noindent\textbf{Case I.} In this case $\g g\cong\gl(m|n)\oplus\gl(m|n)$ and therefore $\bfZ(\g g)\cong
\bfZ(\gl(m|n))\otimes\bfZ(\gl(m|n))$. 
For $i\geq 0$ set $\cA_i:=\sfj\left(\bfZ^{(i)}(\gl(m|n))\otimes 1\right)$ and $\cB_i:=\sfj\left(1\otimes \bfZ^{(i)}(\gl(m|n))\right)$, where
$\sfj:\bfU(\g g)\to\sPD(V)$ is as in \eqref{eq:dfsfj}.
 From\cite[Thm B.1]{SahSalAdv} we obtain \begin{equation}
\label{eq:sfjZZii}
\sfj(\bfZ^{(i)}\left(\g g)\right)=\cA_i=\cB_i=\sPD^{(i)}(V)^{\g g}.
\end{equation}
For $a:=(a_1,\ldots,a_m)\in\C^m$ and $b:=(b_1,\ldots,b_n)\in\C^n$, set 
\begin{equation}
\label{eq:muI}
\mu_{a,b}:=\sum_{i=1}^m
a_i\eps_i+
\sum_{j=1}^n
b_j\delta_j
.
\end{equation}
Then we have
\[
\aomega=\left\{
(-\mu_{a,b},\mu_{a,b})\,:\,
a\in\C^m,\ b\in\C^n\right\}.
\]
From \eqref{eq:sfjZZii} it follows that
\[
\mathsf{res}\left(\HC\left(\bfZ^{(i)}(\g g)\right)\right)
=\mathsf{res}\left(\HC\left(1\otimes \bfZ^{(i)}(\gl(m|n))\right)\right)
.
\]
Now for $k\in\N$ we define $f_k\in\sP^{(k)}(\aomega)$ by 
\[
f_k(\mu_{a,b}):=\sum_{i=1}^m\left(
a_i+\frac{m+1}{2}-\frac{n}{2}-i
\right)^k
+
(-1)^{k-1}
\sum_{j=1}^n\left(
b_j+\frac{m+1}{2}+\frac{n}{2}-j
\right)^k.
\]
From the explicit description of the image of the Harish-Chandra homomorphism of $\gl(m|n)$, given for example in \cite{CWBook},
it follows that 
$
\mathsf{res}\left(\HC\left(1\otimes \bfZ(\gl(m|n))\right)\right)
$ is the $\C$-algebra generated by the $f_k$ for $k\geq 1$.
Furthermore, by graded-surjectivity of the Harish-Chandra homomorphism
of $\gl(m|n)$ (see the proof of  \cite[Thm 2.26]{CWBook}),
 for every $i\geq 0$ we have 
\[
\HC\big(1\otimes \bfZ^{(i)}(\gl(m|n))\big)=
\HC\big(1\otimes \bfZ(\gl(m|n))\big)
\cap\sP^{(i)}(\g h^*).
\]
Consequently,  $f_1,\ldots,f_k\in\mathsf{res}\left(\HC\left(1\otimes \bfZ^{(k)}(\gl(m|n))\right)\right)=
\mathsf{res}\left(\HC\left(\bfZ^{(k)}(\g g)\right)\right)$ for $k\in\N$.
Next observe that for $k\in\N$, 
\[
\left(\tau_J^*\right)^{-1}(f_k)=\sum_{i=1}^m x_i^k+(-1)^{k-1}\sum_{j=1}^n
y_j^k\in\Lambda_J^{(k)}.
\] 
Since $\deg\left(h_k(x,y,1)-\big(\tau_J^*\big)^{-1}(f_k)\right)<k$, equality in \eqref{resHCsseqtauJ} follows from Lemma \ref{lem:htBt-theL}.\\

\noindent\textbf{Case II.} 
For $a:=(a_1,\ldots,a_m)\in\C^m$ and $b:=(b_1,\ldots,b_n)\in\C^n$  
set 
\begin{equation}
\label{eq:muII}
\mu_{a,b}:=
\sum_{i=1}^m
a_i\eps_i+\sum_{j=1}^n b_j(\delta_{2j-1}+\delta_{2j})
.
\end{equation}
Then
\[
\aomega=\left\{\mu_{a,b}\,:\,a\in\C^m,\ b\in\C^n\right\}
,
\] 
and from the explicit description of the Harish-Chandra homomorphism for $\gl(m|2n)$, given for example in \cite{CWBook}, it follows that $\mathsf{res}(\HC(\bfZ(\g g))$ 
is the $\C$-algebra generated by the polynomials
$f_k\in\sP(\aomega)$ for $k\in \N$,
defined as
\begin{align*}
f_k(\mu_{a,b}):&=\sum_{i=1}^m\left(a_i-\frac{m+1}{2}+n+i\right)^k\\
&+(-1)^{k-1}\sum_{j=1}^n\left(b_j-\frac{m+1}{2}-n+2j-1\right)^k+
\left(b_j-\frac{m+1}{2}-n+2j\right)^k.
\end{align*}
Furthermore, by graded-surjectivity of the Harish-Chandra homomorphism of $\g g$ (see the proof of  \cite[Thm 2.26]{CWBook}) we have
$f_1,\ldots,f_k\in\mathsf{res}\left(\HC\left(\bfZ^{(k)}(\g g)\right)\right)$
for $k\in\N$. 
Next observe that
\[
\left(\tau^*_J\right)^{-1}(f_k)=(-2)^k
\left[
\sum_{i=1}^m x_i^k+\frac{(-1)^{k-1}}{2^k}\sum_{j=1}^n
\left(y_j+\frac{1}{2}\right)^k+
\left(y_j-\frac{1}{2}\right)^k
\right]\in\Lambda_J^{(k)}.
\]
Furthermore $\deg\left(
\left(-\frac{1}{2}\right)^k
\left(\tau^*_J\right)^{-1}(f_k)
-h_k(x,y,\frac{1}{2})\right)<k$, so that Lemma
\ref{lem:htBt-theL}
implies equality in \eqref{resHCsseqtauJ}.\\

\noindent \textbf{Case III.}
For $a,b\in\C$ set 
\begin{equation}
\label{eq:muIII}
\mu_{a,b}:=a\eps_1+b\zeta.
\end{equation} Then 
\[
\aomega:=\left\{
\mu_{a,b}\,:\,a,b\in\C
\right\},
\] and from the explicit description of the Harish-Chandra homomorphism for $\g{osp}(m+1|2n)$ (see \cite[Thm 2.26]{CWBook})
it follows that $\mathsf{res}(\HC(\bfZ(\g g)))$ is 
generated by $f_1,f_2\in\sP(\aomega)$,  defined as
\[
f_1(\mu_{a,b}):=b\text{ and }
f_2(\mu_{a,b}):=\left(a+\frac{m+1}{2}-n-1\right)^2.
\]
Furthermore, $f_k\in\mathsf{res}(\HC(\bfZ^{(k)}(\g g)))$ for $k=1,2$ ($f_1$ lies in the image of the center of $\g{gosp}(m+1|2n)$, and $f_2$ lies in the image of the Casimir operator of $\g{osp}(m+1|2n))$. Next observe that 
\[
\left(\tau_J^*\right)^{-1}
\left(f_1
-\frac{m+1}{2}+n+1
\right)=x_1+x_2
\text{ and }
\left(\tau_J^*\right)^{-1}\left(f_2-
\left(f_1
-\frac{m+1}{2}+n+1
\right)^2
\right)=-4x_1x_2.
\]
The statement now follows from the fact that 
$x_1+x_2$ and $x_1x_2$ are algebraically independent generators of 
the algebra of symmetric polynomials in two variables
(see \cite[Sec. I.2]{MacdonaldSymHall}).\\

\noindent\textbf{Case IV.}
For $a,b\in\C$, set \begin{equation}
\label{eq:muIV}
\mu_{a,b}:=a\eps_1+b(\delta_1+\delta_2).
\end{equation} Then 
\[
\aomega=\{\mu_{a,b}\,:\,a,b\in\C\},
\] and 
from the explicit description of the Harish-Chandra 
homomorphism of $\gl(1|2)$
it follows that the $\C$-algebra $\mathsf{res}(\HC(\bfZ(\g g)))$ is generated by $f_k\in\sP^{(k)}(\aomega)$ for $k\in\N$,  where 
\[
f_k(\mu_{a,b}):=(a+1)+(-1)^{k-1}((b-1)^k+b^k).
\]
Setting $\tilde{a}:=\frac{1}{2}(a+1)$ and $\tilde{b}:=b-\frac{1}{2}$, we obtain that in the new coordinates $(\tilde{a},\tilde{b})$, the polynomial $f_k$ is equal to 
\[
\tilde f_k(\tilde{a},\tilde{b}):=2^k\left[
\tilde{a}^k+
\frac{(-1)^{k-1}}{2^k}
\left(
\left(\tilde{b}-\frac{1}{2}\right)^k
+\left(\tilde{b}+\frac{1}{2}\right)^k\right)
\right].
\]
Let $W_k\sseq\sP^{(k)}(\aomega)$ be  defined as 
\[
W_k:=
\spn_\C\left\{f_1^{m_1}\cdots f_k^{m_k}\,:\,
m_j\in\Z^{\geq 0}
\text{ for $j\geq 1$ and }
\sum_{j=1}^k jm_j\leq k\right\}.
\] 
For $k\in\N$, let $h_k(x,y,\frac{1}{2})\in\Lambda_{1,1,\frac{1}{2}}^\natural$ be defined as in 
Lemma \ref{lem:htBt-theL}.
Since
%$\tilde f_k\in\Lambda_{1,1,\frac{1}{2}}^\natural$ and 
\[
\textstyle
\deg\left(\tilde f_k(\tilde{a},\tilde{b})-h_k(\tilde{a},\tilde{b},\frac{1}{2})\right)<k
\ \text{ for }k\in\N,
\] 
Lemma \ref{lem:htBt-theL} implies that 
%$W_d=\Lambda_{1,1,\frac{1}{2}}^{\natural,d}$. 
$\dim W_k=
\left|\EuScript{H}_k(1,1)\right|$. It is straightforward to verify that $\left(\tau_J^*\right)^{-1}(f_k)\in\Lambda_{1,1,-\frac{1}{t}}^{\natural,(k)}$ for $k\in\N$, so that 
$\left(\tau_J^*\right)^{-1}(W_k)\sseq 
\Lambda_{1,1,-\frac{1}{t}}^{\natural,(k)}$. Since $\dim W_k=\left|
\EuScript{H}_k(1,1)
\right|
=\dim\Lambda_{1,1,-\frac{1}{t}}^{\natural,(k)}$, 
we obtain $\left(\tau_J^*\right)^{-1}(W_k)= 
\Lambda_{1,1,-\frac{1}{t}}^{\natural,(k)}$.
This completes the proof of \eqref{resHCsseqtauJ}.\\

\noindent\textbf{Case VI.} 
For $a:=(a_1,\ldots,a_n)\in\C^n$ and 
$b:=(b_1,\ldots,b_n)\in\C^n$ set
$\tilde{\mu}_{a,b}:=
\sum_{i=1}^n(a_i\eps_i+b_i\delta_i)$.
From the explicit description of the image of the Harish-Chandra homomorphism of $\gl(n|n)$ corresponding to $\g b_{n|n}^\mathrm{mx}$
it follows that $\HC\left(\bfZ(\g g)\right)$ is generated by 
the polynomials $\tilde{f}_k\in\sP(\g h^*)$ for $k\geq 1$, where
$\tilde{f}_k(\tilde{\mu}_{a,b}):=
\sum_{i=1}^n(x_i+\frac{1}{2})^r+(-1)^{r-1}
\sum_{i=1}^n(y_i+\frac{1}{2})^r$.
For $k\in\N$ let $h_k(x,y,1)\in\Lambda_{n,n,1}^\natural$ be defined as in Lemma \ref{lem:htBt-theL}. 
Since $\deg(h_k-\tilde{f}_k)<k$, graded surjectivity of the Harish-Chandra homomorphism of $\gl(n|n)$ (see the proof of \cite[Thm 2.26]{CWBook}) 
and 
Lemma \ref{lem:htBt-theL}
imply that 
\begin{equation}
\label{eq:HCCVI}
\textstyle
\HC\left(\bfZ^{(k)}(\g g)\right)=\spn_\C\left\{\tilde{f}_1^{m_1}\cdots
\tilde{f}_k^{m_k}\,:\,
m_j\in\Z^{\geq 0}\text{ for $j\geq 1$ and }
\sum_{j=1}^k jm_j\leq k \right\}.
\end{equation}
Next for $a:=(a_1,\ldots,a_n)\in\C^n$ set 
\begin{equation}
\label{eq:muVI}
\mu_a:=\sum_{i=1}^na_i(\eps_i+\delta_i).
\end{equation}
Then 
\[
\aomega=\left\{\mu_{a}\,:\,a\in\C^n\right\}.
\]
Set $f_k:=\tilde{f}_k\big|_{\aomega}$, so that 
\[
f_k(\mu_a):=
\sum_{i=1}^n\left(a_i+\frac{1}{2}\right)^k
+(-1)^{k-1}\sum_{j=1}^n\left(a_i-\frac{1}{2}\right)^k.
\]
Also, set $p_r(\mu_a):=\sum_{i=1}^n a_i^r$ for $r\geq 0$. 
Then $f_k$ is a linear combination of the $p_{2i-1}$ 
for $1\leq i\leq \lfloor \frac{k+1}{2}\rfloor$. 
Therefore from  \eqref{eq:HCCVI} it follows that 
 \begin{equation}
 \label{msFFRs}
 \mathsf{res}(\HC(\bfZ^{(k)}(\g g)))=
\spn_\C\left\{\prod_{i\in\N}{p}_{2i-1}^{m_{i}}
\,:\,m_j\in\Z^{\geq 0}\text{  for $j\geq 1$  and }\sum_{j\in\N} (2j-1)m_{j}\leq k
 \right\}.
 \end{equation}
The statement of the proposition  is now a consequence of the fact that the right hand side of 
\eqref{msFFRs} is equal to 
$\tau_J^*\left(\Lambda_J^{(k)}\right)$ (this follows for example from
\cite[Thm 2.11]{Pragacz} and \cite[Rem. 2.6]{Pragacz}).\\

%  is generated by the polynomials $f_k\in\sP(\aomega)$ for $k\in\N$,where Clearly each $f_k$ is a linear combination of  the polynomials $\sum_{i=1}^na_i^{2m-1}$ for $1\leq m\leq \lfloor \frac{k+1}{2}\rfloor$. The equality \eqref{resHCsseqtauJ} now follows from the fact that *******************the family of power sums $\sum_{i=1}^nx_i^{2m-1}$, for  $m\geq 1$, are algebraically independent generators of the algebra $\Gamma:=\varprojlim_n\Gamma_n$ of (see \cite[Sec III.8]{MacdonaldSymHall}).\\

\noindent\textbf{Case VII.} 
The argument is similar to Case I, based on the explicit description of the Harish-Chandra homomorphism of $\g q(n)$ (see \cite[Thm 2.46]{CWBook}).
In this case $\aomega:=\{(-\mu_{a},\mu_{a})\,:\,a\in\C^n\}$, where for $a:=(a_1,\ldots,a_n)\in\C^n$ we define
\begin{equation}
\label{eq:muVII}
\mu_a:=\sum_{i=1}^na_i\eps_i.
\qedhere\end{equation}
%\begin{align*}
%\left(\tau_J^*\right)^{-1}(f_k)
%&=\left(
%\frac{1+t}{3+t}y_1-\frac{1+t}{3+t}x_1-2
%\right)^k
%+(-1)^{k-1}
%\left(
%-\frac{1+t}{(3+t)^2}x_1
%-\frac{(1+t)(2+t)}{(3+t)^2}y_1-2
%\right)^k\\
%&+(-1)^{k-1}
%\left(
%-\frac{1+t}{(3+t)^2}x_1
%-\frac{(1+t)(2+t)}{(3+t)^2}y_1-1
%\right)^k
%\end{align*}
\end{proof}

\begin{prp}
\label{prp:typeFresHCsseqtauJ*}
Assume that $J\cong \mathit{F}$. Then 
\[
\mathsf{res}\left(\HC\left(\bfZ(\g g)\right)\right)
\subsetneq
\tau_J^*\left(\Lambda_J\right).
\]
\end{prp}

\begin{proof}
Recall that $\g g\cong\g{gosp}(2|4)$. For $a,b,c\in\C$ set
\begin{equation}
\label{eq:muabc}
\mu_{a,b,c}:=a\eps_1+b\delta_1+b\delta_2+
c\zeta.
\end{equation} 
Then \begin{equation}
\label{eq:muxyzF}
\aomega:=\left\{
\mu_{a,b,c}\,:\,
a,b,c\in\C\right\},
\end{equation} and from the explicit description of the Harish-Chandra homomorphism of $\g{osp}(2|4)$ (see for example \cite[Thm 2.26]{CWBook})
it follows that 
$\mathsf{res}(\HC(\bfZ(\g g)))$ is the $\C$-algebra generated by the following three families of polynomials.
\begin{itemize}
\item[(i)] The polynomials $f_k$ for $k\in\N$, where
$f_k(\mu_{a,b,c}):=
(a+2)^{2k}-(b+2)^{2k}-(b+1)^{2k}$.
\item[(ii)] The polynomials 
\[F_g(\mu_{a,b,c}):=(a+2)
\left(
(b+2)^2-(a+2)^2
\right)
\left(b+1)^2-(a+2)^2
\right)
\left[g\left(a+2,(b+2)^2,(b+1)^2\right)\right],
\]
where $g(\mathsf s,\mathsf t_1,\mathsf t_2)$ is a polynomial in the variables 
$\mathsf s,\mathsf t_1,\mathsf t_2$ which is
symmetric in $\mathsf t_1,\mathsf t_2$.
\item[(iii)]
The polynomial $Q(\mu_{a,b,c}):=c$.
\end{itemize}
By a straightforward calculation based on the above generators we can verify that 
\[
\left(\tau_J^*\right)^{-1}\big(\mathsf{res}(\HC(\bfZ(\g g))\big)\sseq\Lambda_{2,1,\frac{3}{2}}^\natural
.
\]
To complete the proof of the proposition, it suffices to show that  
\begin{equation}
\label{eq:h4notin}
\textstyle h_3(x_1,x_2,y_1,\frac{3}{2})\not\in
\left(\tau_J^*\right)^{-1}\big(\mathsf{res}(\HC(\bfZ(\g g))\big)
, 
\end{equation}
where 
$h_3(x_1,x_2,y_1,\frac{3}{2})\in\Lambda_{2,1,\frac{3}{2}}^\natural$ is defined as in 
Lemma \ref{lem:htBt-theL}.
Consider the change of coordinates
\[
\begin{cases}
\tilde{a}:=a+b+\frac{7}{2}=2x_1+3y_1\\
\tilde{b}:=a-b+\frac{1}{2}=2x_2+3y_1\\
\tilde{c}:=c=x_1+x_2+y_1.
\end{cases}
\]
In the $(\tilde{a},\tilde{b},\tilde{c})$ coordinates, the generators $f_k$, $F_g$, and $Q$ turn into $\tilde f_k$, $\tilde F_g$, and $\tilde Q$, defined below.
\begin{itemize}
\item[(i)] $\tilde f_k(\tilde a,\tilde b,\tilde c):=\frac{1}{2^{2k}}\left[
(\tilde a+\tilde b)^{2k}-(\tilde a-\tilde b+1)^{2k}-(\tilde a-\tilde b-1)^{2k}\right]$.
\item[(ii)] $\tilde F_g(\tilde a,\tilde b,\tilde c):=\frac{1}{32}\left[
(\tilde a+\tilde b)(4\tilde a^2-1)(4\tilde b^2-1)
\right]
g\left(\frac{\tilde a+\tilde b}{2},(\frac{\tilde a-\tilde b+1}{2})^2,(\frac{\tilde a-\tilde b-1}{2})^2\right)$.
\item[(iii)] $\tilde Q(\tilde a,\tilde b,\tilde c):=\tilde c$.
\end{itemize}
Note that the $\tilde{f}_k$ are independent of $c$.
Also, $\tau_J^*\left(h_3(x_1,x_2,y_1,\frac{3}{2})\right)$  can be expressed in $(\tilde a,\tilde b,\tilde c)$-coordinates as
\begin{align}
\label{eq:11over32}
\frac{81}{64}{\tilde c}^3-\frac{135}{128}(\tilde a+\tilde b){\tilde c}^2
&+
\left(\frac{171}{256}(\tilde a^2+\tilde b^2)+\frac{27}{128}\tilde a\tilde b-\frac{51}{64}
\right){\tilde c}\\
&+
\left(
-\frac{53}{512}(\tilde a^3+\tilde b^3)
-\frac{63}{512}(\tilde a^2\tilde b+\tilde a\tilde b^2)
+\frac{35}{128}(\tilde a+\tilde b)
\right)
.
\notag
\end{align}
Now assume that  \eqref{eq:h4notin} is not true. It follows that the polynomial \eqref{eq:11over32}
belongs to the $\C$-algebra generated by  the $\tilde {f}_k$'s, the $\tilde{F}_g$'s, and $\tilde{Q}$. 
Since the variables $\tilde a,\tilde b,\tilde c$ are algebraically independent, the coefficient of $\tilde c^2$ 
in \eqref{eq:11over32}
should belong to the $\C$-algebra generated by 
the $\tilde{f}_k$ and the $\tilde{F}_g$'s.
It follows that there exist polynomials 
 $\phi({u}_1,\ldots,{u}_N)\in\C[u_1,\ldots,u_N]$ for some $N\in\N$ 
and $g_\circ(\tilde a,\tilde b)\in\C[\tilde a,\tilde b]$ 
 such that  
\begin{equation}
\label{eq:a+b=olineF}
\textstyle (\tilde a+\tilde b)=
(\tilde a+\tilde b)(4\tilde a^2-1)(4\tilde b^2-1)
g_\circ(\tilde a,\tilde b)+\phi\left(\tilde{f}_1(\tilde a,\tilde b),\ldots,\tilde{f}_N(\tilde a,\tilde b)\right).
\end{equation}
Setting $\tilde a=\frac{1}{2}$ and $\tilde b=0$ in \eqref{eq:a+b=olineF}, we obtain
\begin{equation}
\label{eq:12=11a}
\textstyle \frac{1}{2}=\phi\left(
-(\frac{3}{2})^2,\ldots,-(\frac{3}{2})^{2N}\right),
\end{equation}
and setting $\tilde a=-\frac{1}{2}$ and $\tilde b=0$
in \eqref{eq:a+b=olineF},
 we obtain 
\begin{equation}
\label{eq:12=11b}
\textstyle 
-\frac{1}{2}=\phi\left(
-(\frac{3}{2})^2,\ldots,-(\frac{3}{2})^{2N}\right).
\end{equation}
Clearly \eqref{eq:12=11a} and \eqref{eq:12=11b}
cannot be true simulteneously. This contradiction implies that 
\eqref{eq:h4notin} must be true. 
\end{proof}

\section{Proof of Theorem 
\ref{thm-main-opVla}
when $J\not\cong\mathit{F}$.}
\label{sec-thm-main-opVla}

The proof of Theorem \ref{thm-main-opVla} differs in the two cases $J\not\cong\mathit{F}$ and 
$J\cong\mathit{F}$. Indeed in the case $J\not\cong \mathit{F}$
Proposition \ref{prp-JresHCZiLJ} allows us to give a uniform proof. Thus, we first prove 
Theorem \ref{thm-main-opVla} in the case $J\not\cong\mathit{F}$, and then we give a separate argument for the case
 $J\cong \mathit{F}$. We begin with the following lemma, whose proof is similar to that of
\cite[Lem. 5.4]{SahSalAdv}.
Recall that $c_\mu(\lambda)$ for $\lambda,\mu\in \Omega$ denotes the scalar by which $D_\mu$ acts on $V_\lambda$.

\begin{lem}
\label{lem:CapDlamVmu}
Let $\mu\in\Omega_d$ where $d\geq 0$. Then 
$c_\mu(\mu)=d!$, and 
$c_\mu(\lambda)=0$ for all $\lambda\in\bigcup_{k=0}^d
\Omega_k\bls\{\mu\}$.

\end{lem}

\begin{proof}
If $\lambda\in\Omega_k$ where $k<d$, then $V_\lambda\sseq\sP^k(V)$, and thus $c_\mu(\lambda)=0$ because $D_\mu\sP^k(V)=\{0\}$.  
Next assume $\lambda\in\Omega_d$. 
The map $\sfm:\sPD(V)\otimes \sP(V)\to\sP(V)$ defined by $D\otimes f\mapsto Df$
and the canonical isomorphism $\sfm:\sP(V)\otimes\sS(V)\to\sPD(V)$
are $\g g$-equivariant. Since 
$D_\mu\in\sfm(V_\mu\otimes V_\mu^*)$, the restriction of 
$D_\mu$ to $V_\lambda$ is a $\g g$-equivariant map $V_\lambda\to V_\mu$, so that $c_\mu(\lambda)=0$ unless $\lambda=\mu$.

Finally, assume $\lambda=\mu$. Let 
$\lag\cdot,\cdot\rag$ be the duality pairing between $\sP^d(V)\cong\sS^d(V)^*$ and $\sS^d(V)$, and let $\beta:\sP^d(V)\times \sS^d(V)\to\C$ be the bilinear map 
$\beta(a,b):=\partial_b(a)$, where \[
\partial:\sS(V)\to \sD(V)
\] is the canonical isomorhpism between the symmetric algebra 
$\sS(V)$
and the algebra of constant coefficient differential operators
$\sD(V)$. A 
direct calculation shows that 
$\beta(\cdot,\cdot)=d!\lag \cdot,\cdot\rag$. Now choose a basis $v_1,\ldots,v_t$ for $V_\mu$ and
a dual basis $v_1^*,\ldots,v_t^*$ for $V_\mu^*$. Then 
\[
\textstyle
D_\mu v_k=\sfm(\sum_{i=1}^tv_i\otimes v_t^*)v_k=\sum_{i=1}^t
\partial_{i=1}^tv_i\partial_{v_i^*}v_k
=d!v_k.\qedhere
\] 
\end{proof}

Recall the maps 
\[
\sfj:\bfU(\g g)\to\sPD(V),\ \
\HC:\bfU(\g g)\to\sP(\g h^*), \text{ and }
\mathsf{res}:\sP(\g h^*)\to\sP(\aomega),
\] defined in 
\eqref{eq:dfsfj}, 
\eqref{eq:HHCC1}, and \eqref{eqdfreS}.
For $D\in\sPD(V)^\g g$
and $\lambda\in\Omega$
let $\oline{\HC}(D)\left(\hww{\lambda}\right)$ denote the scalar by which $D$ acts on 
the irreducible $\g g$-module $V_\lambda\sseq\sP(V)$
whose $\g b$-highest weight is $\hww{\lambda}$.
%%%%%%%%%%%%%%%%%%%%
%By Proposition \ref{prp-g=k+bdegofDlam}, if the order of $D$ is $d$ then $\oline{\HC}_D\in\sP^d(\aomega)$.  Note that the diagram
%\begin{equation}
%\label{eqcommdiagHCbar}
%\begin{gathered}\xymatrix{
%\bfZ(\g g)\ar[r]^{\sfj}  \ar[d]_{\HC}& \sP(V)^{\g g}
%\ar[d]^{\oline{\HC}}
%\\
%\sP(\g h^*) \ar[r]_{\mathsf{res}}& \sP(\aomega) 
%}
%\end{gathered}
%\end{equation}
%is commutative.
%%%%%%%%%%%%%%%%%%%%%%%%%%%%%%%
Then we have
\begin{equation}
\label{eqcommdiagHCbar}
\oline{\HC}(\sfj(z))=\mathsf{res}(\HC(z))
\text{ for }z\in\bfZ(\g g).
\end{equation}

Recall that 
by $\left(\sP^{(i)}_J\right)_{i\geq 0}$
we denote the standard degree filtration of 
the polynomial algebra $\sP_J$ 
defined in
\eqref{eq:LamJPJ}.
Let $\tau_J^{}:\aomega\to \C^{n_J}$ and 
$
\tau_J^*:\sP_J\to \sP\left(\aomega\right)
$ 
be defined as in \eqref{eq:introtauJ} and \eqref{eq:introtauJ2}, respectively.
Since  $\tau_J^{}$ is a bijection,   $\tau_J^*$ is an isomorphism of $\C$-algebras.

If $J\not\cong\mathit{F}$, then by Proposition
\ref{prp-JresHCZiLJ},
for every $\mu\in\Omega_d$ there exists an element $z_\mu\in\bfZ^{(d)}(\g g)$ such that 
\[
\mathsf{res}(\HC(z_\mu))=\tau_J^*\left(P_{J,\mu}\right),
\] 
where $P_{J,\mu}$ is as in Definition 
\ref{dfn:Omega}.
Theorem \ref{thm-main-opVla} follows from Proposition \ref{prp::-surj-center}(iii).

\begin{prp}
\label{prp::-surj-center} 
Assume that  $J\not\cong\mathit{F}$.
Then the following assertions hold.
\begin{itemize}
\item[\rm (i)]
$\sfj\left(z_\mu\right)=D_\mu$
for all $\mu\in\Omega$. 
\item[\rm  (ii)] 
$\sfj(\bfZ(\g g))=\sPD(V)^{\g g}$.

\item[\rm (iii)] $\oline{\HC}_{D_\mu}(\hww{\lambda})=P_{J,\mu}(\tau_J(\hww{\lambda}))$ for all $\lambda,\mu\in\Omega$.
\end{itemize}

\end{prp}

\begin{proof}
(i) By a direct computation, from Theorem \ref{thm:SVSP*} and Theorem \ref{thm-Iv} it follows that 
$P_{J,\mu}\left(\tau_J(\hww{\mu})\right)=d!$ and 
\[
P_{J,\mu}\left(\tau_J(\hww{\lambda})\right)=0
\text{ for all }\lambda\in\bigcup_{k=0}^d
\Omega_k\bls\{\mu\}.
\]  
Set $D_\mu':=\sfj(z_\mu)$. Then 
$D_\mu'\in\sPD^{(d)}(V)^\g g$ because the map $\sfj$ preserves the filtrations. 
From \eqref{eqcommdiagHCbar} it follows that
for $\lambda\in\bigcup_{k=0}^d\Omega_k$, the operator
$D_\mu'$ acts on $V_\lambda$ by the scalar
$P_J\left(\tau_J(\hww{\lambda})\right)$.
Since elements of $\sPD^{(d)}(V)$ are uniquely determined by their restrictions to $\sP^d(V)$, Lemma
\ref{lem:CapDlamVmu} implies that $D_\mu=D_\mu'$.

(ii) Since the family $\left(D_\mu\right)_{\mu\in\Omega}$ is a basis of $\sPD(V)^{\g g}$,  the above argument
implies that  $\sfj(\bfZ(\g g))=\sPD(V)^{\g g}$.

(iii) This follows immediately from the fact that $D_\mu=D_\mu'$.
\end{proof}

We remark that Proposition \ref{prp::-surj-center}(ii) does not hold when $J\cong\mathit{F}$.  This is proved in Proposition \ref{prpJ===Fsjscfd}. 
\begin{prp}
\label{prpJ===Fsjscfd}
Assume that $J\cong F$. Then 
$\sfj\left(\bfZ(\g g)\right)\subsetneq\sPD(V)^{\g g}$
\end{prp}
\begin{proof}
Suppose that 
$\sfj\left(\bfZ(\g g)\right)=\sPD(V)^{\g g}$.
Then from 
\eqref{eqcommdiagHCbar}
and the explicit description of the 
image of the Harish-Chandra homomorphism of $\g{osp}(2|4)$
it follows that 
\begin{equation}
\label{HCbarsseqtauJ*F}
\oline{\HC}\left(\sPD(V)^\g g\right)
=
\oline{\HC}
\left(
\sfj\left(\bfZ(\g g)\right)
\right)
=
\mathsf{res}\left(\HC(\bfZ(\g g))\right)
\sseq \tau_J^*(\Lambda_J).
\end{equation}
Fix $d\geq 0$ and set $\Omega_{\leq d}:=\bigcup_{k=0}^d\Omega_k$.
By Proposition \ref{prp-g=k+bdegofDlam} we have  $\oline{\HC}(D_\lambda)\in\sP^{(d)}(\aomega)$ for $\lambda\in \Omega_{\leq d}$, and therefore from 
\eqref{HCbarsseqtauJ*F} it follows that 
$\oline{\HC}(D_\lambda)\in\tau_J^*\left(\Lambda_J^{(d)}\right)$
for $\lambda\in \Omega_{\leq d}$.
Recall that
elements of $\sPD(V)$ are uniquely determined by their restrictions to $\sP(V)$. 
Since the  
Capelli operators $\left(D_\lambda\right)_{\lambda\in\Omega_{\leq d}}$
are linearly independent, 
 it follows that
 the polynomials $\left(\oline{\HC}(D_\lambda)\right)_{\lambda\in\Omega_{\leq d}}$ are also linearly independent. 
Since $|\Omega_{\leq d}|=\sum_{k=0}^d|\EuScript{H}_k(1,2)|=\dim\Lambda_J^{(d)}$, it follows that the family  $\left(\oline{\HC}(D_\lambda)\right)_{\lambda\in\Omega_{\leq d}}$ also spans $\tau_J^*\left(\Lambda_J^{(d)}\right)$. Since $d$ can be any non-negative integer, from 
\eqref{eqcommdiagHCbar} it follows that \[
\tau_J^*(\Lambda_J)=\bigcup_{d=0}^\infty
\tau_J^*(\Lambda_J^{(d)})\sseq 
\mathsf{res}\left(\oline{\HC}(\sPD(V)^\g g)\right)=
\mathsf{res}\left(\HC(\bfZ(\g g))\right),
\] which contradicts  
Proposition \ref{prp:typeFresHCsseqtauJ*}.  
\end{proof}

\section{Proof of Theorem 
\ref{thm-main-opVla} when $J\cong \mathit{F}$}

\label{Sec:10dimcase}

Recall that in this case $\g g\cong\g{gosp}(2|4)$, and 
$\aomega$ is the subspace of $\g h^*$ that is given in \eqref{eq:muxyzF}.
Let $\sigma:\C^3\to \aomega$ be defined by $\sigma(a,b,c):= \mu_{a,b,c}$, where $\mu_{a,b,c}$ is defined in 
\eqref{eq:muabc}. We define the map $\sigma^*:\sP(\aomega)\to\sP(\C^3)\cong \C[a,b,c]$ by
$\sigma^*(f):=f\circ \sigma_F$.
\begin{lem}
\label{lemDG/KnotdpzZ}
Assume that $J\cong \mathit{F}$. Then $\sigma^*(\tau_J^*(\Lambda_J))$ is the subalgebra of $\C[a,b,c]$ consisting of polynomials
$f(a,b,c)$ which satisfy 
the following two properties.
\begin{itemize}
\item[\rm (i)] $f(a,b,c)=f(a,-b-3,c)$.

\item[\rm (ii)]
$f(a+1,b+\frac{1}{2},c)=f(a-1,b-\frac{1}{2},c)$ on the affine hyperplane $a-b+\frac{1}{2}=0$.

\end{itemize}
 
\end{lem}

\begin{proof}
This is straightforward from the explicit description of $\tau_J$ given in Table \ref{tbl-3}.
\end{proof}
Now for  $D\in\sPD(V)^\g g$ we have 
$\sigma^*\left(\oline\HC(D)\right)\in\C[a,b,c]$. 
 We have a direct sum decomposition $\g g\cong \g g'\oplus\g z$ where $\g g':=[\g g,\g g]\cong\g{osp}(2|4)$ and $\g z:=\g z(\g g)\cong \C$ is the centre of $\g g$. 
 
From now on, let $\cG$ denote the simply connected complex Lie supergroup
corresponding to $\g g$, and let $\cV$ be the affine superspace corresponding to $V$. The $\g g$-action on $V$ lifts to a $\cG$-action on $\cV$. Set $\cK:=\mathrm{stab}_{\cG}^{}(e)$ (see Appendix \ref{AppndxC}). 
 Since 
 $\g z$ is $\cK$-invariant, 
\[
\bfU(\g g)^{\cK}\cong\big(
\bfU(\g g')\otimes \bfU(\g z)\big)^{\cK}\cong 
\bfU(\g g')^{\cK}\otimes \bfU(\g z).
\]  
Remark \ref{rmk:embedP(V)inGK} implies that 
every $D\in\sPD(V)^{\g g}$ can be realized as an element of
$\sD(\cG/\cK)$.
Let $\Psi_{\cG,\cK}$ be defined as in 
\eqref{eq:UgtoDG/K}.
Since 
there is a $\cK$-invariant complement of $\g k$ in $\g g$,
Proposition \ref{prp:UguK->DG/K} implies that 
every $D\in \mathscr D(\cG/\cK)$ lies in  $\Psi_{\mathcal G,\mathcal K}\left(\bfU(\g g)^{\cK}\right)$. 
\begin{lem}
\label{lem:DPsiUg')}

Assume that $J\cong \mathit{F}$. Let $D\in \sPD(V)^\g g$. 
\begin{itemize}
\item[\rm (i)] If $D\big|_{\cG/\cK}\in
\Psi_{\mathcal G,\mathcal K}\left(1\otimes \bfU(\g z)\right)$, then $\sigma^*\left(\oline\HC(D)\right)\in\C[c]$.
\item[\rm (ii)]
 If $D\big|_{\cG/\cK}\in
\Psi_{\mathcal G,\mathcal K}\left(\bfU(\g g')^{\cK}\otimes 1\right)$,
then $\sigma^*\left(\oline\HC(D)\right)\in\C[a,b]$.  
\end{itemize}

\end{lem}

\begin{proof}
The proof of (i) is straightforward. For (ii), let 
$\cG'$ be the subsupergroup of $\cG $ corresponding to $\g g'\sseq\g g$. For $\tilde X\in \bfU(\g g)^\cK$, let $\mathrm{L}_{\tilde X}$ denote the left invariant differential operator induced by $\tilde X$ on 
$\cG/\cK$ (see 
Appendix \ref{AppndxC}). Similarly, any 
$\tilde X\in \bfU(\g g')^{\cK}$ induces a left invariant differential operator on 
$\cG'/\cK$, which we denote by $\mathrm{L}_{\tilde X}$ as well.
For $\tilde X\in \bfU(\g g')^{\cK}\sseq\bfU(\g g)^{\cK}$, the diagram
\[
\xymatrix@C=+20mm{
\ccO_{\cG/\cK}(G/K)
\ar[r]^{\phi\mapsto \phi|_{\cG'/\cK}}
\ar[d]_{\mathrm{L}_{\tilde X}} & \ccO_{\cG'/\cK}(G'/K)\ar[d]^{\mathrm{L}_{\tilde X}}\\
\ccO_{\cG/\cK}(G/K)\ar[r]_{\phi\mapsto \phi|_{\cG'/\cK}} & \ccO_{\cG'/\cK}(G'/K)
}
\]
is commutative. 
According to Remark \ref{rmk:embedP(V)inGK}, there is a $\cG$-equivariant embedding 
\[
\sfp_e^\#(V_\eev)\big|_{\sP(V)}:\sP(V)\into 
\ccO_{\cG/\cK}(G/K).
\] For $\lambda\in\Omega$, let $p_\lambda\in V_\lambda\sseq \sP(V)$ be a $\g b$-highest weight vector and 
set 
$\phi_\lambda:=\sfp_e^\#(V_\eev)(p_\lambda)$.
Since $\phi_\lambda\neq 0$ and $p_\lambda$ is a homogeneous element of $\sP(V)$, we obtain $\phi_\lambda\big|_{\cG'/\cK}\neq 0$. 
Set $\g h':=\g h\cap\g g'$. 
From Proposition \ref{prp-g=k+bdegofDlam} 
and the fact that $\g g'=\g k+(\g b\cap \g g')$
it follows that $D$ acts on $V_\lambda$ is by the scalar $\hww\lambda(\tilde D_{\g h'})$ for
some  $\tilde D_{\g h'}\in \bfU(\g h')$ that only depends on $D$. Since the map 
$\hww\lambda\mapsto \hww\lambda(\tilde D_{\g h'})$ only depends on $\hww\lambda\big|_{\g h'}$, it is indeed an element of 
$\sP(\g {a'}^*)$, where 
$\g {a'}^*:=\left\{\mu_{a,b,0}\,:\,a,b\in\C\right\}$, for
$\mu_{a,b,c}$ defined as in \eqref{eq:muabc}.
 Finally, to complete the proof observe that $\sigma^*\left(\sP(\g {a'}^*)\right)=\C[a,b]$.
\end{proof}

\begin{prp}
\label{prp:ab-b-3}
Let $D\in\sPD(V)^{\g g}$. If $D\big|_{\cG/\cK}
\in\Psi_{\cG,\cK}(\bfU(\g g')^\cK\otimes 1)$, then  
$h:=\sigma^*\left(\oline\HC(D)\right)$ satisfies the relation $h(a,b)=h(a,-b-3)$.
\end{prp}

\begin{proof}
Let us denote the one-dimensional $\g{so}(2)$-module with  weight $k\eps_1$ by $M(k\eps_1)$, and the irreducible $\g{sp}(4)$-module with 
$\g b\cap\g{sp}(4)$-highest weight   
$k_1\delta_1+k_2\delta_2$ by $M'(k_1\delta_1+k_2\delta_2)$. We have 
$V^*\cong V_\eev^*\oplus V_\ood^*$, where
\[
V_\eev^*\cong 
M(\eps_1)\otimes M'(\delta_1+\delta_2)\oplus M(3\eps_1)\otimes M'(0)\ \text{ and }\ V_\ood^*\cong  M(2\eps)\otimes M'(\delta_1),
\]
as $\g{so}(2)\oplus\g{sp}(4)$-modules.
Set \[
U^*:=M(\eps_1)\otimes M'(\delta_1+\delta_2)\ \text{ and }\ W^*:=M(3\eps_1)\otimes M'(0)\oplus  M(2\eps)\otimes M'(\delta_1).
\] 
Then $V^*\cong U^*\oplus W^*$, 
from which we obtain a natural tensor product decomposition 
\begin{equation}
\label{sPVsPsDDD}
\sP(V)\cong\sP(U)\otimes \sP(W).
\end{equation}
Using \eqref{sPVsPsDDD}, we identify 
$\sP(U)$ and $\sP(W)$ with subalgebras 
$\sP(U)\otimes 1$ and $1\otimes \sP(W)$ of 
$\sP(V)$.
By dualizing the relation $V^*\cong U^*\oplus W^*$ 
we obtain a direct sum decomposition $V\cong U\oplus W$ for subspaces $U,W$ of $V$. The latter direct sum decomposition yields a tensor product decomposition 
\begin{equation}
\label{sPVsPsDDD1}
\sD(V)\cong\sD(U)\otimes \sD(W),
\end{equation}
which allows us to identify $\sD(U)$ and $\sD(W)$ with subalgebras of $\sD(V)$.

It is straightforward to verify that  $U^*=(V^*)^{\g b_\ood}$. Thus  every $\g b_\eev$-invariant vector in $\sP(U)\subset \sP(V)$ is also a $\g b$-highest weight vector in $\sP(V)$. 
By Lemma \ref{sp4lem-SC}, we have an $\g{so}(2)\oplus 
\g{sp}(4)$-module isomorphism 
\[
\sP^k(U)\cong\bigoplus_{i=0}^{\lfloor\frac{k}{2} \rfloor}M(k\eps_1)\otimes M'\big((k-2i)\delta_1+(k-2i)\delta_2\big).
\]
Recall that $\mu_{a,b,c}$ denotes
the element of $\aomega$ defined in
\eqref{eq:muabc}. 
The Zariski closure of the set \[
S:=\left\{
k\eps_1+(k-2i)\delta_1+(k-2i)\delta_2\,:\,
i,k\in\Z^{\geq 0}\text{ and }0\leq i\leq \lfloor\frac{k}{2}\rfloor\right\}
\]
is equal to $\{\mu_{a,b,0}\,:\,a,b\in\C\}$.
%Therefore $h(a,b)$ is uniquely determined by its values on the set . 

Let $\sP^+(W)$ and $\sD^+(W)$ denote the augmentation ideals of $\sP(W)$ and $\sD(W)$, respectively. 
Using \eqref{sPVsPsDDD}, we obtain a decomposition
\begin{equation}
\label{sPDVcUoplusW}
\sPD(V)\cong \sPD(U)\oplus\sPD(U)^\perp,
\end{equation}
where 
\[
\sPD(U)^\perp:=\sPD(V)\sD^+(W)+\sP^+(W)\sPD(U)
.
\]
Write 
$D:=D_U\oplus D_U^\perp$, where $D_U\in\sPD(U)$ and $D_U^\perp\in\sPD(U)^\perp$.  
For $0\leq i\leq \lfloor\frac{k}{2}\rfloor$, choose a nonzero vector \[
v_{k,i}\in M(k\eps_1)\otimes M'\big((k-2i)\delta_1+(k-2i)\delta_2\big)
.\] Then $Dv_{k,i}=c_{k,i} v_{k,i}$ for some scalar $c_{k,i}\in\C$, hence
$D_U^\perp v_{k,i}=Dv_{k,i}-D_Uv_{k,i}\in\sP(U)$. Since $\sD^+(W)v_{k,i}=0$, we obtain $D_U^\perp v_{k,i}\in\sP(U)\cap \left(\sP^+(W)\sP(V)\right)=\{0\}$, so that $Dv_{k,i}=D_Uv_{k,i}$. Since 
the decomposition \eqref{sPDVcUoplusW} is $\g g_\eev$-invariant, we have   $D_U\in\sPD(U)^{\g g_\eev}$. Therefore from \cite[Sec. 11.4]{HoweUmeda} it follows that $D_U$ lies in the algebra generated by the degree operator and the image of the Casimir operator of $\g{sp}(4)$. 
Let $f_1,f_2\in\sP(\aomega)$ denote the 
eigenvalues of the degree and Casimir operators, respectively, and set $h_i:=\sigma^*(f_i)$ for $i=1,2$.
The degree operator acts on $v_{k,i}$ by the scalar $k$. It follows that  $h_1(a,b)=a$, and therefore  $h_1(a,b)=h_1(a,-b-3)$.
Similarly, the Casimir operator of $\g{sp}(4)$ acts on $v_{k,i}$ by the scalar $(k+2)^2+(k+1)^2$, so that $h_2(a,b):=(b+2)^2+(b+1)^2$, and therefore $h_2(a,b)=h_2(a,-b-3)$.
The statement of the proposition follows from the fact that
$h$ belongs to the subalgebra
of $\C[a,b]$ generated by $h_1$ and $h_2$. \end{proof}

\begin{prp}
\label{prp:a+1b+12}
Let $D\in\sPD(V)^{\g g}$. If $D\big|_{\cG/\cK}
\in\Psi_{\cG,\cK}(\bfU(\g g')^\cK\otimes 1)$, then  
$h:=\sigma^*\left(\oline\HC(D)\right)$ satisfies the relation
$h(a+1,b+\frac{1}{2})=h(a-1,b-\frac{1}{2})$ for all $a,b\in\C$ such that $a-b+\frac{1}{2}=0$.
\end{prp}

\begin{proof}
Let 
$\g b_{2|4}\sseq\g g$ be the 
Borel subalgebra defined as in
\eqref{bm|2ngosp}.
Note that $\g b_{2|4}\neq \g b$.
For $\lambda\in\Omega$, let $\underline{\underline{\lambda}}$ denote the $\g b_{2|4}$-highest weight of $V_\lambda$. 
The Borel subalgebra $\g b$ can be obtained from $\g b_{2|4}$ by the composition of the odd reflections 
given in \eqref{eq:bbexoddref}.
Thus from \cite[Lem. 1.40]{CWBook} it follows that if $\hww{\lambda}$
is one of the typical highest weights of the form given in \eqref{d+2s-4eps1+bd}, 
then 
$\underline{\underline{\lambda}}:
=\hww{\lambda}+4\eps_1$, whereas
if $\hww{\lambda}=
d\eps_1+d\delta_1+d\delta_2+d\zeta$, then $\underline{\underline{\lambda}}
=\hww{\lambda}+2\eps_1-\delta_1-\delta_2$. 
In particular, the
$\g b_{2|4}$-highest weight $\underline{\underline{\lambda}}$
always lies in $\aomega$, where 
$\aomega$ is the subspace of $\g h^*$  given in \eqref{eq:muxyzF}.

Now for every $\g b_{2|4}$-highest weight $\underline{\underline{\lambda}}\in\aomega$, let $f(\underline{\underline{\lambda}})$ denote the scalar by which $D$ acts on $V_\lambda$. 
We have $\g g=\g b_{2|4}+\g k^\mathrm{ex}$, and thus
by Proposition \ref{prp-g=k+bdegofDlam} the map 
$\underline{\underline{\lambda}}\mapsto f(\underline{\underline{\lambda}})$ 
can be extended to an element of $\sP(\aomega)$.  
Set $h_1:=\sigma^*(f)$.
From  Lemma \ref{lem:DPsiUg')}(ii) it follows that indeed $h_1\in\C[a,b]$.
Since the eigenvalue of $D$ on $V_\lambda$ is independent of the choice of the Borel subalgebra, from typical highest weights we  obtain \begin{equation}
\label{eq:h1eq(1))}
h_1(a+4,b)=h(a,b)\text{ for }a,b\in\C,
\end{equation}
 and from the atypical highest weights, which correspond to the partitions $\lambda:=(d,0,\ldots)$,
we obtain 
\begin{equation}
\label{eq:h1eq(2))}
h_1(a+2,a-1)=h(a,a)
\text{ for }a\in\C.
\end{equation} 
Consequently, if $a-b+\frac{1}{2}=0$
then $a+1=b+\frac{1}{2}$ and thus
from 
\eqref{eq:h1eq(1))} and \eqref{eq:h1eq(2))} it follows that
\[
\textstyle 
h(a+1,b+\frac{1}{2})=h_1(a+3,b-\frac{1}{2})
=h(a-1,b-\frac{1}{2}).
\qedhere
\]
\end{proof}

\subsection{Proof of Theorem \ref{thm-main-opVla} when
$J\cong\mathit{F}$. }
Fix $\mu\in\Omega_d$ for $d\geq 0$.
Then $D_\mu\in\sPD^{(d)}(V)^\g g$ and therefore by Proposition
\ref{prp-g=k+bdegofDlam}
 there exists $h\in\sP^d(\aomega)$ such that for every $\lambda\in\Omega$, the operator $D_\mu$ acts on $V_\lambda$ 
by the scalar
$h(\hww{\lambda})$. 
By Lemma \ref{lem:CapDlamVmu} we have
\begin{equation}
\label{eq:hhwwmuhla=00}
h(\hww{\mu})=d!\text{ and }
h(\hww{\lambda})=0\text{ for }
\lambda\in\bigcup_{k=0}^d \Omega_k\bls\{\lambda\}.
\end{equation}
Let $f\in\sP_J^{(d)}$ be defined by $f:=\left(\tau_J^*\right)^{-1}
\left(h\right)$.
From 
Lemma \ref{lemDG/KnotdpzZ}, Lemma \ref{lem:DPsiUg')}, Proposition
\ref{prp:ab-b-3}, and Proposition 
\ref{prp:a+1b+12} it follows that 
$f\in\Lambda_J^{(d)}$. Furthermore,
\begin{equation}
\label{eq:tauJlam>=2}
\textstyle
\tau_J(\hww{\lambda})
=\left(
\lambda_1+\frac{1}{4},
\lambda_2-\frac{5}{4},d-\lambda_1-\lambda_2+1
\right)
=
\left(\sfp_1(\lambda),\sfp_2(\lambda),\sfq_1(\lambda)\right),
\end{equation}
where $\sfp_1(\lambda)$, $\sfp_2(\lambda)$, and $\sfq_1(\lambda)$ are defined as in \eqref{eq:piqiFrb} for $m=2$, $n=1$, and $\theta=\frac{3}{2}$. 
From \eqref{eq:hhwwmuhla=00} and  \eqref{eq:tauJlam>=2} it follows that $f$ and $P_{J,\mu}$ satisfy the same degree, symmetry, and vanishing properties, so that $f=P_{J,\mu}$ by Theorem \ref{thm:SVSP*}.

\appendix

\section{The TKK construction}
\label{Sec-TKKconst}

Recall that a vector superspace $J:=J_\eev\oplus J_\ood$ is called a 
\emph{Jordan superalgebra} if it is equipped with a supercommutative bilinear product $J\times J\to J$ which satisfies the Jordan identity
\[
(-1)^{|x||z|}[L_x,L_{yz}]+(-1)^{|y||x|}[L_y,L_{zx}]+(-1)^{|z||y|}[L_z,L_{xy}]=0
\ \text{ for homogeneous $x,y,z\in J$,}
\]
where  we define 
 $L_a:J\to J$  for $a\in J$ to be  the left multiplication map $x\mapsto ax$, and denote
the parity of  a homogeneous element $a\in J$ by $|a|$.

%Throughout this article,  $J:=J_\eev\oplus J_\ood$ will be   a  finite dimensional complex simple Jordan superalgebra.

Following \cite{CanKac}, by a \emph{short grading} of  a Lie superalgebra $\g l$ we mean a $\Z$-grading of $\g l$ of the form 
$\g l:=\bigoplus_{t\in\Z}\g l(t)$,
such that $\g l(t)=\{0\}$ for $t\not\in\{0,\pm 1\}$.
Using the Kantor functor, 
in \cite{KacJ}   
Kac  
associates to $J$ a simple Lie superalgebra
$\gJJ$ 
(the \emph{TKK Lie superalgebra})
with a short grading 
\[
\gJJ:=
\gJJ(-1)\oplus\gJJ(0)\oplus\gJJ(1).
\]
We recall the definition of $\gJJ$. Set
$\gJJ(-1):=J$, $\gJJ(0):=\spn_\C \{L_a,[L_a,L_b]\,:\,a,b\in J\}\sseq \mathrm{End}_\C(J)$, and 
$
\gJJ(1):=\spn_\C\{ P,[L_a,P]\,:\,a\in J\}\sseq\mathrm{Hom}_\C(\sS^2(J),J),
$ 
where $P:\sS^2(J)\to J$ is the map
$P(x,y):=xy$, and  $
[L_a,P](x,y):=a(xy)-(ax)y-(-1)^{|x||y|}(ay)x
$.
The Lie superbracket of $\gJJ$ is defined by the following relations.
\begin{itemize}
\item[(i)] $[A,a]:=A(a)$ for $A\in \gJJ(0)$ and $a\in \gJJ(-1)$.
\item[(ii)] $[A,a](x):=A(a,x)$ for $A\in\gJJ(1)$, $a\in \gJJ(-1)$, and $x\in J$. 
\item[(iii)] $[A,B](x,y):=A(B(x,y))-(-1)^{|A||B|}B(A(x),y)-(-1)^{|A||B|+|x||y|}B(A(y),x)$ for $A\in \gJJ(0)$, $B\in\gJJ(1)$, and $x,y\in J$.

\end{itemize}
 For the classification of finite dimensional  complex simple Jordan superalgebras and their corresponding TKK Lie superalgebras, see the articles by Kac \cite{KacJ} and
Cantarini--Kac  \cite{CanKac}.

If $J$ has a unit  $1_J\in J$, then the elements $e:=1_J$, $f:=-2P$, and $h:=2L_{1_J}$ of $\gJJ$ satisfy 
\eqref{Eqhefrel}.
It follows that 
$\g s:=\spn_\C\{e,f,h\}$ is a subalgebra of $\gJJ$ isomorphic to $\g{sl}_2(\C)$.
Indeed $\g s$ is  a \emph{short subalgebra} of $\gJJ$ (see \cite{CanKac}). We recall the definition of a short subalgebra.
\begin{dfn}
\label{shortsubalg}
Let $\g l$ be a complex Lie superalgebra. A 
\emph{short subalgebra}  of $\g l$ is a Lie subalgebra $\g a\sseq \g l_\eev$ that is isomorphic to $\g{sl}_2(\C)$, with  a basis $e,f,h$ that satisfies the relations
\eqref{Eqhefrel}, 
such that the eigenspace decomposition of $\mathrm{ad}\left(-\frac{1}{2}h\right)$ defines a short grading of $\g l$.
\end{dfn}

\begin{rmk}
\label{remonA(mn)case}
Let $\g l$ be a complex Lie superalgebra and let $\g a\sseq\g l_\eev$ be a  short subalgebra of $\g l$. 
\begin{itemize}
\item[(a)]
Assume that $\g l$ is a subalgebra of another Lie superalgebra $\tilde{\g l}$ such that 
 $\dim\tilde{\g l}=\dim\g l+1$. 
Since every finite dimensional $\g{sl}_2(\C)$-module is completely reducible, it follows that $\tilde{\g l}\cong \g l\oplus \C$ as $\g a$-modules, so that  $\g a$ is a short subalgebra of $\tilde{\g l}$ as well.  

\item[(b)] Every  central extension 
$0\to \C\to \hat{\g l}\to \g l\to 0$
splits on $\g a$.  An argument similar to part (a) implies that the image of $\g a$ under the splitting section is a short subalgebra of $\hat{\g l}$.

%\item[(c)] Assume that $J\cong\mathit{gl}(m,n)_+$, where $\mathit{gl}(m,n)_+$ denotes the Jordan superalgebra of complex $(m|n)\times (m|n)$-matrices. In this case,if $m\neq n$ then $\gJJ\cong\g{sl}(2m|2n)$, and if $m=n$, then $\gJJ\cong\g{psl}(2m|2n)$.Using (a) and (b), we can lift $\g s$ to a short subalgebra of $\gl(2m|2n)$. The advantage o$\gJJ(0)$ is isomorphic to a codimension one subalgebra of $\gl(m|n)\oplus\gl(m|n)$, $\gJJ(0)$ is isomorphic to a quotient of the latter Lie superalgebra. However, since we are interested in the action of $\gJJ(0)$ on $\gJJ(-1)$, it will be more convenient to work with $\gl(m|n)\oplus\gl(m|n)$.To this end, we use parts (a) and (b) to obtain a short subalgebra of $\gl(2m|2n)$.
\end{itemize}

% and therefore it will be more convenient to work with the Lie superalgebra $\gl(2m|2n)$, which does not arise inthe TKK construction. However, the short subalgebra $\g s$ naturally corresponds to a short subalgebra  of $\gl(2m|2n)$,  because $\g{sl}(2m|2n)\subset \gl(2m|2n)$, and the central extension \[0\to\C \to \g{sl}(2m|2m)\to \g{psl}(2m|2m)\to 0\] splits over $\g s$.We will denote the latter short subalgebra of $\gl(2m|2n)$ by  $\g s$ as well. Similarly, when $J\cong\mathit{q}(n)_+$ (respectively, $J\cong\mathit{p}(n)_+$),we can lift the short subalgebra of $\gJJ$ to a short subalgebra of $\g{q}(2n)$ (respectively, $\g{p}(2n)$), which we will denote by $\g s$ as well. 
\end{rmk}
When $J$ is isomorphic to $\mathit{gl}(m,n)_+$, $\mathit{p}(n)_+$, or $\mathit{q}(n)_+$, 
it will be more convenient for us 
to replace $\gJJ$ by a non-simple Lie superalgebra
which has a more natural matrix realization (see also Remark \ref{rmk-reasongflat}). To this end, we define the Lie superalgebra $\gflat$ by
\begin{equation}
\label{EqdfgJtild}
\gflat:=
\begin{cases}
\g{gl}(2m|2n)&\text{ if }
J\cong\mathit{gl}(m,n)_+,\\
\g{p}(2n)&\text{ if }J\cong\mathit{p}(n)_+,\\
\g{q}(2n)&\text{ if }J\cong\mathit{q}(n)_+,\\
\gJJ&\text{ otherwise.}
\end{cases}
\end{equation} 
 For a precise description of $\g{p}(2n)$ and $\g{q}(2n)$
see Appendix \ref{appendix-classical}.
%The reason for replacing $\gJJ$ by $\gflat$ is explained in Remark \ref{rmk-reasongflat}.
From Remark \ref{remonA(mn)case} it follows that the short subalgebra $\g s$ of $\gJJ$ corresponds to a unique
short subalgebra of  $\gflat$.  
We use the same symbols $\g s$, $e$, $f$, and $h$ for denoting the short subalgebra of $\gflat$ and its corresponding basis.

By restriction of the adjoint representation, $\gflat$ is equipped with an
 $\g s$-module structure. This $\g s$-module is a direct sum of trivial and adjoint representations of $\g s$, hence it integrates to a representation of the adjoint group $\mathrm{PSL}_2(\C)$. 
%The elements $\mathsf E:=f$, $\mathsf H:=2h$, and $\mathsf F:=-2e$ form a standard $\g{sl}_2$-triple, and therefore 
Furthermore, \begin{equation}
\label{involw}
w:=\exp(\ad({f}))
\exp(-\ad({e}))
\exp(\ad({f}))
\end{equation}
represents the nontrivial element of the Weyl group of $\mathrm{PSL}_2(\C)$.
%This element satisfies Ad(w)h=-h,  Ad(w)e=1/2f, and Ad(w)f=2e.
%%%
%%%From old to new: H=-2h, E=e,F=-2f.
%%%

Set
 $\gflat(t):=\{x\in\gflat\, :\, [h,x]=-2tx\}$ for $t\in\{0,\pm 1\}$. 
The Lie superalgebra  $\gflat(0)$ naturally acts on $\gflat(-1)\cong J$. Set
\begin{equation*}
%\label{Eqggtil0k}
\g g:=\gflat(0)
\text{ and }
V:=\gflat(-1)^*:=\Hom_\C(V,\C^{1|0}).
\end{equation*}
Thus the $\g g$-module $V$ is the dual of the $\g g$-module $J$.

\begin{rmk}
\label{rmk-reasongflat}
The reason for replacing $\gJJ$ by $\gflat$ is to obtain a convenient way of associating partitions to the irreducible $\g g$-modules that occur in $\sP(V)$. For example, assume that $J\cong\mathit{gl}(m,n)_+$, where
$\mathit{gl}(m,n)_+$ denotes the Jordan superalgebra of $(m+n)\times (m+n)$ matrices in $(m,n)$-block form. Then $\gJJ\cong \g{sl}(2m|2n)$ if $m\neq n$, and    
$\gJJ\cong \g{psl}(2m|2n)$ if $m=n$. In both cases, $\gJJ(0)$ is closely related to $\gl(m|n)\oplus \gl(m|n)$, but it is \emph{not} isomorphic to it. However, $\g g:= \gflat(0)\cong \gl(m|n)\oplus \gl(m|n)$, and the irreducible summands of $\sP(V)\cong\sP((\C^{m|n})^*\otimes \C^{m|n})$
 are naturally parametrized by  $(m,n)$-hook partitions.
\end{rmk}

\section{Classical Lie superalgebras}
\label{appendix-classical}
In this Appendix we give explicit realizations of classical Lie superalgebras $\g{gl}(m|n)$, 
$\g{gosp}(m|2n)$, $\g p(n)$, and $\g q(n)$. We describe root systems of 
$\g{gl}(m|n)$, 
$\g{gosp}(m|2n)$,  and $\g q(n)$, and choose Borel subalgebras in these Lie superalgebras.
% In the case of   $\g{gl}(m|n)$ and $\g q(n)$, for each chosen Borel subalgebra we describe a map from partitions into the set of corresponding highest weights.    
\subsection{The Lie superalgebra $\gl(m|n)$} 
\label{subsec-gl}
Let $m,n\geq 1$ be integers. We use the usual realization of $\gl(m|n)$ as 
$(m+n)\times (m+n)$ matrices in $(m,n)$-block form 
\begin{equation}
\label{bldiaABCD}
\begin{bmatrix}
A& B\\
C& D
\end{bmatrix}
\end{equation}
where $A$ is $m\times m$ and $D$ is $n\times n$. 
The diagonal Cartan subalgebra of  $\gl(m|n)$ is
\begin{equation}
\label{eq:CSAhm|n}
\g h_{m|n}:=\left\{\mathrm{diag}(\mathbf s,\mathbf t)\,:\,\mathbf{s}:=(s_1,\ldots, s_m)\in\C^m\text{ and }\mathbf{t}:=(t_1,\ldots,t_n)\in\C^n\right\}.
\end{equation}
The standard characters $\eps_i,\delta_j:\g h_{m|n}\to \C$ are  
defined by 
\[
\eps_i(\mathrm{diag}(\mathbf s,\mathbf t)):=s_i\text{ for }1\leq i\leq m\ \text{ and }\  
\delta_j(\mathrm{diag}(\mathbf s,\mathbf t)):=t_j\text{ for }1\leq j\leq n.
\]
We define $\g b^{\mathrm{st}}_{m|n}$ (respectively, $\g b^{\mathrm{op}}_{m|n}$)
to be the Borel subalgebras of 
$\gl(m|n)$ corresponding to the fundamental systems
$\mathit{\Pi}^{\mathrm{st}}$ (respectively, $\mathit{\Pi}^{\mathrm{op}}$), where
\[
\mathit{\Pi}^{\mathrm{st}}:=\left\{\eps_i-\eps_{i+1}\right\}_{i=1}^{m-1}\cup\left\{\eps_m-\delta_1\right\}
\cup\left\{\delta_j-\delta_{j+1}\right\}_{j=1}^{n-1}
\text{ and }
\mathit{\Pi}^{\mathrm{op}}:=-\mathit{\Pi}^{\mathrm{st}}
\]
%and\[\Pi^{\mathrm{op}}:=\left\{\eps_{i+1}-\eps_{i}\right\}_{i=1}^{m-1}\cup\left\{\delta_1-\eps_m\right\}\cup\left\{\delta_{i+1}-\delta_{i}\right\}_{j=1}^{n-1}.
%\]
For every partition $\lambda\in \EuScript{H}(m,n)$, we set 
\begin{equation}
\label{eq:lammn-stt}
\lambda^{\mathrm{st}}_{m|n}:=
\sum_{i=1}^m\lambda_i\eps_i+
\sum_{j=1}^n\lag\lambda_j'-m\rag\delta_j.
\end{equation}

In the spacial case $m=n$, 
we define $\g b_{n|n}^{\mathrm{mx}}
$
to be the Borel subalgebra corresponding to the fundamental system
\[
\mathit{\Pi}^{\mathrm{mx}}
:=\left\{
\delta_i-\eps_i
\right\}_{i=1}^n\cup
\left\{\eps_{j}-\delta_{j-1}
\right\}_{j=2}^{n}.
\]

\subsection{The Lie superalgebra $\g{gosp}(m|2n)$}
\label{subsec-osp}
Let $m,n\geq 1$ be integers.
We begin with an explicit realization of $\g{osp}(m|2n)$.  
Set $r:=\lfloor\frac{m}{2}\rfloor$.
Let $J^+$ be the $m\times m$ matrix defined by
\[
J^+:=\begin{bmatrix}
1& 0_{1\times r} & 0_{1\times r}\\
0_{r\times 1} & 0_{r\times r} & I_{r\times r}\\
0_{r\times 1}& I_{r\times r} & 0_{r\times r}
\end{bmatrix}
\text{ if $m=2r+1$, and }
J^+:=\begin{bmatrix}
0_{r\times r} & I_{r\times r}\\
I_{r\times 1}& 0_{r\times r} 
\end{bmatrix}
\text{ if $m=2r$,}
\]
Also, let $J^-$ be the $2n\times 2n$ matrix defined by
\[
J^-:=\begin{bmatrix}
0_{n\times n}& I_{n\times n}\\
-I_{n\times n} & 0_{n\times n}
\end{bmatrix}.
\]
Let $\{\sfe_i\}_{i=1}^m\cup\{\sfe'_j\}_{j=1}^{2n}$ be the standard homogeneous basis of $\C^{m|2n}$, and let
$
\mathsf{B}:\C^{m|2n}\times \C^{m|2n}\to\C 
$
be the even supersymmetric bilinear form
defined by 
\[
\mathsf{B}(\sfe_i,\sfe_j)=J^+_{i,j},\
\mathsf{B}(\sfe_i',\sfe_j')=J^-_{i,j},
\text{ and }
\mathsf{B}(\sfe_i,\sfe_j')=0.
\]
%whose matrix in this basis  is\begin{equation}\label{eq-mtrxofB}\begin{bmatrix}J_+&0_{m\times 2n}\\0_{2n\timesm} & J_-\end{bmatrix}.\end{equation}
We realize the Lie superalgebra $\g{osp}(m|2n)$ as the  subalgebra of $\gl(m|2n)$ that leaves the bilinear form 
$\mathsf B:\C^{m|2n}\times \C^{m|2n}\to\C$ invariant. For $\mathbf s\in\C^r$ and $\mathbf t\in\C^n$, set
\[
\mathsf d(\mathbf s,\mathbf t):=
\begin{cases}
\mathrm{diag}(\mathbf s,-\mathbf s,\mathbf t,-\mathbf t)&\text{ if }m=2r,\\
\mathrm{diag}(0,\mathbf s,-\mathbf s,\mathbf t,-\mathbf t)&\text{ if }m=2r+1.
\end{cases}
\]
Recall from Appendix \ref{subsec-gl}
that we denote the standard Cartan subalgebra of $\gl(m|2n)$ by $\g h_{m|2n}$. Then  $\oline{\g h}_{m|2n}:=\g h_{m|2n}\cap\g{osp}(m|2n)$ 
is a Cartan subalgebra of $\g{osp}(m|2n)$.
We have 
\[
\oline{\g h}_{m|2n}=\left\{\mathsf d(\mathbf s,\mathbf t)\,:\,\mathbf s\in\C^r\text{ and }\mathbf t\in\C^n\right\},
\] and 
the standard characters
of $\oline{\g h}_{m|2n}$ 
are given by 
\[
\eps_i(\mathsf d(\mathbf s,\mathbf t)):=s_i
\text{ for }1\leq i\leq r
\text{ and }
\delta_j(\mathsf d(\mathbf s,\mathbf t)):=t_j
\text{ for }1\leq j\leq n.
\] 
Let $\breve{\g b}_{m|2n}\sseq\g{osp}(m|2n)$ be the Borel subalgebra corresponding to the  fundamental system $\mathit{\Pi}$, where
\[
\mathit{\Pi}
:=
\begin{cases}
\left\{
\eps_i-\eps_{i+1}
\right\}_{i=1}^{r-1}
\cup\left\{\eps_r-\delta_1\right\}
\cup\left\{
\delta_j-\delta_{j+1}
\right\}_{j=1}^{n-1}
\cup\left\{
\delta_n
\right\}&\text{ if }m=2r+1,\\[2mm]
\left\{
\eps_i-\eps_{i+1}
\right\}_{i=1}^{r-1}
\cup
\left\{\eps_r-\delta_1\right\}
\cup
\left\{
\delta_j-\delta_{j+1}
\right\}_{j=1}^{n-1}
\cup\left\{
2\delta_n
\right\}&\text{ if }m=2r.
\end{cases}\]
Finally, we set $\g{gosp}(m|2n):=\g{osp}(m|2n)\oplus\C I\sseq\gl(m|2n)$, where $I:=I_{(m+2n)\times (m+2n)}$. 
We also set
\begin{equation}
\label{bm|2ngosp}
\g b_{m|2n}:=\breve{\g b}_{m|2n}\oplus\C I.
%\sseq\g{gosp}(m|2n). 
\end{equation}
We extend the standard characters 
$\eps_i,\delta_j$
of $\oline{\g h}_{m|2n}$ 
to 
the subalgebra 
$\tilde{\g h}_{m|2n}:=\g h_{m|2n}\cap\g{gosp}(m|2n)$
of diagonal matrices in $\g{gosp}(m|2n)$,
by setting $\eps_i(I)=\delta_j(I)=0$.
Let $\zeta:\tilde{\g h}_{m|2n}\to\C$ be the linear functional defined uniquely by \[
\zeta\big|_{\oline{\g h}_{m|2n}}=0\text{ and }\zeta(I)=1.
\]
The set $\{\eps_i\}_{i=1}^r\cup\{\delta_j\}_{j=1}^n\cup \{\zeta\}$ is a basis for the dual of $\tilde{\g h}_{m|2n}$. 
We remark that when  $\g g$ is of type $\g{gosp}$ (i.e., in Cases III and V), 
we have $\zeta(h)=2$ where 
$h\in\gflat$ is   defined as in \eqref{Eqhefrel}.

In Case V, where $\g g=\g{gosp}(2|4)$, we need to consider an exceptional Borel subalgebra 
\begin{equation}
\label{b2|4ex}
\g b_{2|4}^\mathrm{ex}:=
\hat{\g b}^{}_{2|4}\oplus \C I,
\end{equation} where 
$\hat{\g b}^{}_{2|4}$
is the Borel subalgebra of $\g{osp}(2|4)$  corresponding  to the 
fundamental system 
\[
\mathit{\Pi}^{\mathrm{ex}}:=
\{-\eps_1-\delta_1,\delta_1-\delta_2,2\delta_2\}.
\]

\subsection{The anisotropic embedding of $\g{osp}(m|2n)$ in $\gl(m|2n)$}
\label{composp}
We will need another realization of $\g{osp}(m|2n)$ inside $\gl(m|2n)$ which will be used in the description of the spherical subalgebra $\g k$. 
Set $r:=\lfloor\frac{m}{2}\rfloor$, and let $\tilde{J}^-$ be 
the $2n\times 2n$ matrix defined by
\[
\tilde{J}^-:=
\mathrm{diag}
(\underbrace{\tilde{J},\ldots,\tilde{J}}_{n\text{ times}})
\text{ where }
\tilde{J}:=\begin{bmatrix}
0 & 1\\ -1&0\end{bmatrix}
.
\]
Let $\tilde{\mathsf{B}}:\C^{m|2n}\times \C^{m|2n}\to \C$ be the even supersymmetric bilinear form which is given
in the standard basis $\{\sfe_i\}_{i=1}^m\cup
\{\sfe_j\}_{j=1}^{2n}$ of $\C^{m|2n}$ by
\[
\tilde{\mathsf{B}}(\sfe_i,\sfe_j):=\delta_{i,j},\
\tilde{\mathsf{B}}(\sfe'_i,\sfe'_j):=\tilde{J}^-_{i,j},
\text{ and }
\tilde{\mathsf{B}}(\sfe_i,\sfe_j'):=0.
\]
Thus, the matrix of $\tilde{\mathsf{B}}(\cdot,\cdot)$ in the standard basis
of $\C^{m|2n}$ is
\[\begin{bmatrix}I_{m\times m}& 0_{m\times 2n}\\0_{2n\times m}& \tilde{J}^-\end{bmatrix}.\]
The subalgebra of $\gl(m|2n)$ that leaves the bilinear 
form  $\tilde{\mathsf{B}}$ invariant 
is isomorphic to $\g{osp}(m|2n)$. 

\subsection{The exceptional  embedding of $\g{osp}(1|2)\oplus\g{osp}(1|2)$ in $\g{gosp}(2|4)$}
\label{appBexk}
We consider the realization of  $\g{osp}(2|4)$ given in 
Appendix  \ref{subsec-osp}. 
Set\[
g:=\begin{bmatrix}
0 & \sqrt{-1} & 0 & 0 & 0 & 0\\
-\sqrt{-1} & 0 & 0 & 0 & 0 & 0\\
0 & 0 & 0 & 0 & 0 & -\sqrt{-1}\\
0 & 0 & 0 & 0 & \sqrt{-1} & 0\\
0 & 0 & 0 &  -\sqrt{-1} & 0 & 0\\
0 & 0 &   \sqrt{-1} & 0 & 0 & 0
\end{bmatrix}.
\]
We
set $\g k^\mathrm{ex}$ to be the subalgebra of fixed points of the map $\g{osp}(2|4)\to\g{osp}(2|4)$ given by  $x\mapsto \Ad_g(x)$.
One can verify that $\g k^\mathrm{ex}\cong\g{osp}(1|2)\oplus\g{osp}(1|2)$. 
We will consider 
$\g k^\mathrm{ex}$ as a subalgebra 
of $\g{gosp}(2|4)$.

\subsection{The Lie superalgebra $\g p(n)$}
\label{subsec-p}
Let $n\geq 1$ be an integer, and let 
$\check{\mathsf{B}}:\C^{n|n}\times \C^{n|n}\to\C$ be the odd supersymmetric bilinear form defined by  
\[
\check{\mathsf{B}}(\sfe_i,\sfe_j'):=\delta_{i,j},\
\check{\mathsf{B}}(\sfe_i,\sfe_j):=0,
\text{ and }
\check{\mathsf{B}}(\sfe'_i,\sfe'_j):=0,
\]
where $\{\sfe_i\}_{i=1}^n
\cup\{\sfe_i'\}_{i=1}^n$ is the standard homogeneous basis of $\C^{n|n}$. 
The Lie superalgebra $\g p(n)$ is the subalgebra of $\gl(n|n)$ that leaves $\check{\mathsf{B}}(\cdot,\cdot)$
invariant.  It consists of matrices in $(n,n)$-block form
\[
\begin{bmatrix}
A&B\\ C & -A^T
\end{bmatrix},\text{ where }B=B^T\text{ and }C=-C^T.
\]
In this paper we will not need a description of the root system and highest weight modules of $\g p(n)$.

\subsection{The Lie superalgebra $\g q(n)$}
\label{subsec-q}
Let $n\geq 1$ be an integer. 
The Lie superalgebra $\g q(n)$ is the subalgebra of $\gl(n|n)$ that consists of matrices in $(n,n)$-block form
\[
\begin{bmatrix}
A&B\\ B & A
\end{bmatrix}.
\]
Let $\g h$ be the subalgebra of matrices of the latter form where $A$ and $B$ are diagonal. Then $\g h$ is a Cartan subalgebra of $\g q(n)$.  The standard characters $\left\{\eps_i\right\}_{i=1}^n$ of $\g h_\eev$ are the restrictions of the corresponding standard characters of $\gl(n|n)$.  Let $\g b^{\mathrm{st}}_n$ (respectively, $\g b^{\mathrm{op}}_n$) be the Borel subalgebra of 
$\g{q}(n)$ associated to the fundamental system
$
\mathit{\Pi}^{\mathrm{st}}:=
\{\eps_i-\eps_{i+1}\}_{i=1}^{n-1}$ (respectively, 
$
\mathit{\Pi}^{\mathrm{op}}:=
\{\eps_{i+1}-\eps_{i}\}_{i=1}^{n-1}$).
For every partition 
$\lambda\in\EuScript{DP}(n)$, we set
$\lambda^{\mathrm{st}}_n:=\sum_{i=1}^n\lambda_i\eps_i$.

%For every partition $\lambda\in\EuScript{DP}_{n,k}$, we set\[\lambda^{\mathrm{st}}_n:=\sum_{i=1}^n\lambda_i\eps_i\ \text{ and }\ \lambda^{\mathrm{op}}_n:=\sum_{i=1}^n\lambda_{n+1-i}\eps_i.\]We define $E_\lambda$ to be the irreducible $\g q(n)$-module with $\g b^{\mathrm{st}}$-highest weight $\lambda^{\mathrm{st}}_n$. The $\g b^{\mathrm{op}}$-highest weight of $E_\lambda$ is $\lambda^{\mathrm{op}}_n$.

\section{Facts from supergeometry}
\label{AppndxC}

All of the supermanifolds that are considered in this appendix are complex analytic. We denote the underlying complex manifold of a supermanifold $\mathcal{X}$ by $|\mathcal{X}|$, and the sheaf of superfunctions on $\mathcal{X}$ by $\ccO_{\mathcal X}$. 
Morphisms of supermanfolds are expressed as
$(\mathsf{f},\mathsf{f}^\#):\mathcal X\to\mathcal Y$,
where $\mathsf{f}:|\mathcal X|\to|\mathcal Y|$ is the complex analytic map between the underlying spaces and 
$\mathsf{f}^\#:\ccO_{\mathcal Y}\to 
\mathsf{f}_*\ccO_{\mathcal X}$ is the associated morphism of sheaves of superalgebras.

Let $\mathcal L$ be a connected Lie supergroup and let $\mathcal M$ be a Lie subsupergroup of $\mathcal L$. Set $\g l:=\Lie(\mathcal L)$ and $\g m:=\Lie(\mathcal M)$. 
The right action of $\mathcal L$ on $\mathcal L$ induces
 a canonical isomorphism of superalgebras from  $\bfU(\g l)$ onto the algebra of  left invariant holomorphic differential operators on $\mathcal L$. Under this isomorphism elements of  $\bfU(\g l)^\mathcal M$, the subalgebra of $\mathcal M$-invariants in $\bfU(\g l)$, are mapped to holomorphic differential operators which are left $\mathcal L$-invariant and  right $\mathcal M$-invariant. The latter differential operators 
induce $\mathcal L$-invariant differential operators on the homogeneous space $\mathcal L/\mathcal M$. Consequently, we obtain a 
homomorphism of superalgebras
\begin{equation}
\label{eq:UgtoDG/K}
\Psi_{\mathcal L,\mathcal M}:\bfU(\g l)^{\mathcal M}\to \mathscr{D}(\mathcal L/\mathcal M),
\end{equation}
where $\mathscr{D}(\mathcal L/\mathcal M)$ denotes the algebra of $\mathcal L$-invariant differential operators on $\mathcal L/\mathcal M$.
By a superization of the argument of  
\cite[Prop. 9.1]{MeinrenkenClif}, we obtain the following statement. 
\begin{prp}
\label{prp:UguK->DG/K}
Let $\mathscr D^{(d)}(\mathcal L/\mathcal M)$ denote the subspace of elements of 
$\mathscr D(\mathcal L/\mathcal M)$ of order at most $d$. Assume that there exists an $\mathcal M$-invariant complement of $\g m$ in $\g l$. Then 
\[
\Psi_{\mathcal L,\mathcal M}\left(\bfU^{(d)}(\g l)^\mathcal M\right)=
\mathscr D^{(d)}(\mathcal L/\mathcal M)
\text{ for every $d\geq 0$.}
\] 
\end{prp}

In the rest of this appendix  we will assume that $J$ is a Jordan superalgebra of type $\mathsf{A}$. Let $\g g$, $\g k$, and $V$ be as in Section 
\ref{sec:introduction}.

\begin{lem}
There is a vector $v_\g k\in V_\eev$ such that $\g k=\mathrm{stab}_\g g(v_\g k)$. 
\end{lem}
\begin{proof}
In all of the cases where $J$ is of type $\mathsf{A}$ we have 
$V\cong V^*\cong J$ as $\g k$-modules, hence we can set $v_\g k$ equal to the element of $V_\eev$ corresponding to $1_J\in J$.  
\end{proof}

\begin{lem}
\label{gesurjj}
The map $\g g\to V$, $x\mapsto x\cdot v_{\g k}$ is surjective.
\end{lem}
\begin{proof}
The kernel of the linear map $x\mapsto x\cdot v_\g k$ is $\g k$. The statement now follows in all of the cases by verifying that the graded dimension of the image of this map and of $V$ are the same.
\end{proof}
Let $\tilde{\g b}:=\tilde{\g h}\oplus\tilde{\g n}$ be a Borel subalgebra of $\g g$ such that $\g g=\tilde{\g b}+\g k$.  
Let $\cG$ be a complex Lie supergroup such that $\Lie(\cG)=\g g$, and let $\cV$ be the complex affine superspace corresponding to $V$.  
We assume that $|\cG|$ is a connected Lie group, and that the action of $\g g$ on $V$ can be globalized to 
an action of $\cG$ on $\cV$.
The stabilizer of $v_\g k\in |\cV|= V_\eev$ is a complex Lie supergroup $(\cK,\ccO_\cK)$ such that  $\Lie(\cK)=\g k$.

\begin{prp}
\label{prp:orbitmapGK}
The orbit map of $v_\g k$ factors through an embedding 
$(\sfp_{v_\g k},\sfp_{v_\g k}^\#):\cG/\cK\into \cV$ whose image is an open subsupermanifold of $\cV$.
\end{prp}
\begin{proof}
This follows from the fact that the differential of the orbit map 
$(\sfp_{v_\g k},\sfp_{v_\g k}^\#)$
is a bijection for all $g\in \cG$, which is a consequence of 
Lemma \ref{gesurjj}
and $\cG$-equivariance of $(\sfp_{v_\g k},\sfp_{v_\g k}^\#)$.
 \end{proof}

\begin{rmk}
\label{rmk:embedP(V)inGK}
Using the embedding $\cG/\cK\into \cV$ 
of
Proposition \ref{prp:orbitmapGK} and the
natural injection
$\sP(V)\into\ccO_{\cV}(|\cV|)$,
we obtain a $\cG$-equivariant embedding
\[
\sfp_{v_\g k}^\#(|\cV|)\big|_{\sP(V)}:
\sP(V)\into \ccO_{\cG/\cK}\left(|\cG/\cK |\right).
\]
Furthermore, connectedness of $|\cG|$ implies $\sPD(V)^\g g=\sPD(V)^\cG$. 
Therefore we can restrict every $D\in \sPD(V)^{\g g}$ to the open subsupermanifold $\cG/\cK$ of $\cV$, and  
indeed $D\big|_{\cG/\cK}\in\sD(\cG/\cK)$.
\end{rmk}

For the next proposition, recall that every linear functional $\varphi:\tilde{\g h}\to\C$ induces a natural homomorphism of $\C$-algebras 
$\sS\big(\tilde{\g h}\big)\to\C$ given by $x_1\cdots x_k\mapsto\varphi(x_1)\cdots\varphi(x_k)$
for $x_1,\ldots,x_k\in\tilde{\g h}$. 
Set $\sS^{(d)}(\tilde{\g h}):=\bigoplus_{i=0}^d\sS^i\big(\tilde{\g h}\big)$.
\begin{prp}
\label{prp-g=k+bdegofDlam}
Assume that $\sP(V)$ is a completely reducible and multiplicity-free $\g g$-module, and let $D\in\sPD^{(d)}(V)^{\g g}$. Then there exists an element $x_D\in\sS^{(d)}\big(\tilde{\g h}\big)$ such that for every irreducible $\g g$-module $W\sseq \sP(V)$, the action of $D$ on $W$ is by the scalar $\tilde\lambda(x_D)$, where $\tilde\lambda$ is the $\tilde{\g b}$-highest weight of $W$. 

\end{prp}
\begin{proof}
Let $\cG$, $\cV$, and $\cK$ be defined as above, and let $(\sfq,\sfq^\#):\cG\to \cG/\cK$ be the canonical quotient map. Then $\sfq^\#(|\cG/\cK|):\ccO_{\cG/\cK}(|\cG/\cK|)\to \ccO_{\cG}(G)$ is an injection. For any $\tilde{X}\in\bfU(\g g)$, let  $\mathrm{L}_{\tilde X}$
(respectively,  $\mathrm{R}_{\tilde X}$) denote the action of $\tilde X$ on $\ccO_\cG(G)$ by left invariant (respectively,  right invariant) differential operators.  
By Proposition 
\ref{prp:UguK->DG/K},
 there exists $\tilde D\in\bfU^{(d)}(\g g)^{\cK}$ such that
$\sfq^\#(|\cG/\cK|)(Df)=
\mathrm{L}_{\tilde D} \sfq^\#(|\cG/\cK|)(f)$ for 
every  $f\in\ccO_{\cG/\cK}(|\cG/\cK|)$.

Now set $f:=\sfp_{v_\g k}^\#(|V|)(\phi_{\tilde \lambda})$ where 
$\phi_{\tilde \lambda}\in \sP(V)$ is a highest weight vector of $W$,
and let $\tilde f:=\sfq^\#(|\cG/\cK|)(f)$. Let $\cN$ be the connected Lie subsupergroup of $\cG$ such that
$\Lie(\cN)=\tilde{\g n}$.
Then $\tilde f$ is left $\cN$-invariant  and right $\cK$-invariant.
We can express $\tilde D$ in the form  
$\tilde D=D_1+D_2+D_3$,
where \[
D_1\in 
\tilde{\g n}\bfU^{(d-1)}(\g g),
\ 
D_2\in 
\bfU^{(d)}\big(\tilde{\g h}\big),
\text{ and }
D_3\in
\bfU^{(d-1)}(\g g)\g k.
\]
From $\cK$-invariance of $f$ it follows that $\mathrm{L}_{D_3}\tilde f=0$. 
Furthermore,  we can write
$D_1$ as 
a sum of elements of the form $XD'$ where $X\in\tilde{\g n}$ and $D'\in\bfU(\g g)$. 
Let $\cH:=(H,\ccO_{\cH})$ denote the connected Lie subsupergroup of $\cG$ such that
$\Lie(\cH)=\tilde{\g h}$.
For
$h\in H$ we have
\[
\mathrm{L}_{XD'}\tilde f(h)
=\mathrm{L}_X\left(
\mathrm{L}_{D'}\tilde f
\right)(h)=
\mathrm{R}_{-\Ad_hX}
\left(
\mathrm{L}_{D'}\tilde f
\right)(h).
\]
 Since 
$\mathrm{L}_{D'}\tilde f$ is left $\cN$-invariant
and  $\Ad_h\left(\tilde{\g n}\right)\sseq \tilde{\g n}$, it follows that
$\mathrm{L}_{XD'}\tilde f(h)=0$. Consequently, 
we have shown that 
for $x_D:=D_2$,
\[
\mathrm{L}_{\tilde D}\tilde f(h)=\mathrm{L}_{D_2}\tilde f(h)=\tilde{\lambda}\big(D_2\big)\tilde f(h)
=\tilde\lambda(x_D)\tilde f(h)
.
\] It remains to prove that $\tilde f\neq 0$. 
 To this end, note that the canonical multiplication morphism \[
\cN\times \cH\times \cK\to \cG
\] is a local isomorphism at the identity element. Thus, from analyticity, left $\cN$-invariance, and right $\cK$-invariance of $\tilde f$, it follows that $\tilde f\big|_{H}$ is not identically zero. 
\end{proof}

\bibliographystyle{plain}
\bibliography{Capelli-Jordan}

\end{document}